\documentclass[11pt,a4paper,reqno]{amsart}
\usepackage[all]{xy}
\usepackage{amsmath,amsfonts,amssymb,amscd}
\usepackage{soul,color}

\theoremstyle{plain}
\newtheorem{theorem}{Theorem}%[section]
\newtheorem{mytheorem}{Theorem}[subsection]
\newtheorem{corollary}[mytheorem]{Corollary}%[subsection]
\newtheorem{lemma}[mytheorem]{Lemma}%[section]
\newtheorem{proposition}[mytheorem]{Proposition}%[section]
\newtheorem{definition}[mytheorem]{Definition}%[section]
\newtheorem{remark}[mytheorem]{Remark}

\newtheorem{example}[mytheorem]{Example}

\font\cyr wncyr10 at 11pt \def\Sha{\hbox{\cyr X}}

\def\NN{{\mathbb N}}

\def\QQ{{\mathbb Q}}
\def\RR{{\mathbb R}}
\def\ZZ{{\mathbb Z}}

\def\E{{\mathcal E}}

\def\K{{\mathcal K}}
\def\M{{\mathcal M}}

\def\cO{{\mathcal O}}
\newcommand{\fA}{\mathfrak{A}}

\newcommand{\fC}{\mathfrak{C}}

\newcommand{\fE}{\mathfrak{E}}
\newcommand{\fF}{\mathfrak{F}}

\newcommand{\fa}{\mathfrak{a}}
\newcommand{\fb}{\mathfrak{b}}
\newcommand{\fc}{\mathfrak{c}}
\newcommand{\fd}{\mathfrak{d}}
\newcommand{\fe}{\mathfrak{e}}
\newcommand{\ff}{\mathfrak{f}}

\newcommand{\fk}{\mathfrak{k}}

\newcommand{\fm}{\mathfrak{m}}
\newcommand{\fn}{\mathfrak{n}}

\newcommand{\fp}{\mathfrak{p}}

\newcommand{\fr}{\mathfrak{r}}
\newcommand{\fs}{\mathfrak{s}}

\newcommand{\La}{\operatorname{\Lambda}}
\newcommand{\lra}{\longrightarrow}

\newcommand{\lr}{{\longrightarrow\;}}
\newcommand{\mbb}{\mathbb}

\def\+{{\dagger}}
\DeclareMathOperator{\Ker}{Ker}
\DeclareMathOperator{\Coker}{Coker}

\DeclareMathOperator{\Aut}{Aut}

\DeclareMathOperator{\coh}{H}

\DeclareMathOperator{\tr}{tr}

\DeclareMathOperator{\Gal}{Gal}

\DeclareMathOperator{\GL}{GL}

\DeclareMathOperator{\image}{Im}
\DeclareMathOperator{\Sel}{Sel}

\DeclareMathOperator{\Hom}{Hom}
\DeclareMathOperator{\Nm}{N}
\DeclareMathOperator{\Tor}{Tor}

\DeclareMathOperator{\coker}{coker}

\DeclareMathOperator{\rank}{rank}
\DeclareMathOperator{\corank}{corank}

\definecolor{purple}{rgb}{0.7,0,1}

\begin{document}
\title[ ] {Pontryagin duality for Iwasawa modules \\ and abelian varieties}

\author[Lai] {King Fai Lai}
\address{School of Mathematical Sciences\\
Capital Normal University\\
Beijing 100048, China}
\email{kinglaihonkon@gmail.com}

\author[Longhi] {Ignazio Longhi}
\address{Department of Mathematical Sciences\\
Xi'an Jiaotong-Liverpool University\\
No.111 Ren'ai Road, Suzhou Dushu Lake Higher Education Town, Suzhou Industrial Park, Jiangsu, China}
\email{Ignazio.Longhi@xjtlu.edu.cn}

\author[Tan]{Ki-Seng Tan}
\address{Department of Mathematics\\
National Taiwan University\\
Taipei 10764, Taiwan}
\email{tan@math.ntu.edu.tw}

\author[Trihan]{Fabien Trihan}
\address{Department of Information and Communication Sciences\\
Faculty of Science and Technology, Sophia University\\
4 Yonbancho, Chiyoda-ku, Tokyo 102-0081 JAPAN}
\email{f-trihan-52m@sophia.ac.jp}



\begin{abstract}
We prove a functional equation for
two projective systems of finite abelian $p$-groups, $\{\fa_n\}$ and $\{\fb_n\}$, endowed with an action of $\ZZ_p^d$ such that $\fa_n$ can be identified with the Pontryagin dual of $\fb_n$ for all $n$.

Let $K$ be a global field.
Let $L$ be a $\ZZ_p^d$-extension of $K$ ($d\geq 1$), unramified outside a finite set of places.
Let $A$ be an abelian variety over $K$.
We prove an algebraic functional equation for the Pontryagin dual of the Selmer group  of
$A$.
\end{abstract}

\maketitle

\section{Introduction}  \label{sec:intro}

Let $\Gamma$ be an abelian $p$-adic Lie group isomorphic to $\ZZ_p^d$ where $\ZZ_p$ is the ring of
$p$-adic integers and $d$ is a positive integer. The Iwasawa algebra is the complete group ring
$\ZZ_p[[\Gamma]]$ which we denote by $\Lambda$. An  Iwasawa module is a topological $\Lambda$ module.

We study a Pontryagin duality for Iwasawa modules which are inverse limits of finite  Iwasawa modules with $\Lambda$ acting through $\Gamma_n$ where $\Gamma_n$ denotes $\Gamma/\Gamma^{p^n}$. The result leads to a functional equation for characteristic ideals. We then applied these results to Selmer groups of abelian varieties over $\ZZ_p^d$ extensions of global fields.

Before we describe our Pontryagin duality we recall a few simple notions concerning
Iwasawa modules. A $\Lambda$-module $M$ is said to be \textit{pseudo-null} if  no height one prime ideal contains its annihilator (\cite[\S4]{bou65}). A \textit{pseudo-isomorphism} of
$\Lambda$-modules is a homomorphism with pseudo-null kernel and cokernel. We  write $M\sim N$ to mean that $M$ is pseudo-isomorphic to $N$.

The group of roots of unity $\boldsymbol\mu_{p^{\infty}}:=\cup_m\boldsymbol\mu_{p^{m}}$
We say $f\in\Lambda$ is a \textit{simple element} if there exist $\gamma\in\Gamma- \Gamma^p$ and $\zeta\in\boldsymbol\mu_{p^\infty}$ such that
$$f=f_{\gamma,\zeta}:=\prod_{\sigma\in\Gal(\QQ_p(\zeta)/\QQ_p)} (\gamma-\sigma(\zeta)).$$

The inversion $\Gamma\to \Gamma:\gamma\mapsto \gamma^{-1}$ gives rise to an isomorphism
from $\Lambda$ to  $\Lambda$ which we denote by
sending an element $\lambda$ to $\lambda^\sharp$. This allows us to twist a $\Lambda$-module
$M$ to $\Lambda {}_\sharp\!\otimes_{\Lambda} M$ which we denote by $M^\sharp$
(see $\S$\ref{su:twistmodule}).

Now let us describe the formal structure we need for the  Pontryagin duality.

Consider a collection
$$\fA=\{\fa_n,\fb_n,\langle\;,\;\rangle_n,\fr_m^n,\fk_m^n\;\mid\; n,m\in\NN\cup\{0\},\; n\geq m\}$$
where
\begin{enumerate}
\item[($\Gamma$-1)]  $\fa_n, \fb_n$ are finite abelian groups, with an action of $\La$ factoring through $\ZZ_p[\Gamma_n]$.
\item[($\Gamma$-2)] For $n\geq m$,
$$\fr_m^n\colon\fa_m\times \fb_m\longrightarrow \fa_n\times \fb_n\,,$$
$$\fk_m^n\colon\fa_n\times \fb_n\longrightarrow \fa_m\times \fb_m$$
    are $\Gamma$-morphisms such that $\fr_m^n(\fa_m)\subset \fa_n$, $\fr_m^n(\fb_m)\subset \fb_n$, $\fk_m^n(\fa_n)\subset \fa_m$, $\fk_m^n(\fb_n)\subset \fb_m$ and $\fr^n_n=\fk^n_n={\rm id}$. Also, $\{\fa_n\times \fb_n, \fr_m^n\}_n$ form an inductive system and $\{\fa_n\times \fb_n, \fk_m^n\}_n$ form a projective system.
\item[($\Gamma$-3)] We have
$$\fr_m^n\circ\fk_m^n=\Nm_{\Gamma_n/\Gamma_m}\colon\fa_n\times \fb_n\longrightarrow \fa_n\times \fb_n$$
(where $\Nm_{\Gamma_n/\Gamma_m}:=\sum_{\sigma\in\Ker(\Gamma_n\rightarrow\Gamma_m)}\sigma$ is the norm associated with $\Gamma_n\twoheadrightarrow\Gamma_m$) and
$$\fk_m^n \circ \fr_m^n=p^{d(n-m)}\cdot \text{id}\colon\fa_m\times \fb_m\longrightarrow \fa_m\times \fb_m.$$
\item[($\Gamma$-4)] For each $n$, $\langle\;,\;\rangle_n\colon\fa_n\times \fb_n\longrightarrow \QQ_p/\ZZ_p$ is a perfect pairing (and hence $\fa_n$ and $\fb_n$ are dual $p$-groups) respecting $\Gamma$-action as well as the morphisms $\fr_m^n$ and $\fk_m^n$ in the sense that
$$\langle \gamma\cdot a, \gamma\cdot b\rangle_n=\langle a,b\rangle_n\;\forall\,\gamma\in\Gamma,$$
\begin{equation} \label{e:rkpairing} \langle a,\fr_m^n(b)\rangle_n=\langle\fk_m^n(a),b\rangle_m \end{equation}
    and
$$\langle\fr_m^n(a),b\rangle_n=\langle a,\fk_m^n(b)\rangle_m.$$
\end{enumerate}
Write
$$\fa:=\varprojlim_{n}\fa_n\;\;
\text{and}\;\;\fb:=\varprojlim_{n}\fb_n\,.$$

 We say $\fA$ as above is a \textit{$\Gamma$-system} if both $\fa$ and $\fb$ are finitely generated torsion $\La$-modules.  We say that the $\Gamma$-system $\fA$ is \textit{twistable} of order $k$ if there exists an integer $k$ such that $p^{n+k}\fa_n=0$ for every $n$.
Given a $\Gamma$-system $\fA$ we put
$$\fa^0\times \fb^0:=\varprojlim_{m}  \bigcup_{n\geq m}\Ker(\fr_m^{n}).$$
We say $\fA$ is \textit{pseudo-controlled}  if $\fa^0\times \fb^0$ is pseudo-null.

The following Pontryagin duality theorem %on the {\em algebraic functional equation}
is proved in $\S$\ref{se:al}.

\begin{theorem}\label{t:al}
Let $$\fA=\{\fa_n,\fb_n,\langle\;,\;\rangle_n,\fr_m^n,\fk_m^n\;\mid\; n,m\in\NN,\; n\geq m\}$$ be a pseudo-controlled $\Gamma$-system.
Then there is a pseudo-isomorphism
$$\fa^\sharp\sim \fb$$
in the following three cases:\begin{enumerate}
\item there exists $\xi\in\La$ not divisible by any simple element and such that $\xi\fb$ is pseudo-null;
\item $\fA$ is pseudo-isomorphic to a twistable pseudo-controlled $\Gamma$-system;
\item $\fA$ is part of a {\em{\textbf{T}}}-system.
\end{enumerate}
\end{theorem}
The definition of \textbf{T}-system is given in $\S$\ref{su:Ts}.

In $\S$ 4 we introduce Cassels-Tate systems which are constructed
from Selmer groups of abelian varieties defined over global fields.
In \S\ref{ss:gamma}
we show in theorem  \ref{t:ct} that a Cassels-Tate system  is actually a $\Gamma$-system.
In $\S$\ref{sb:fleq} we apply our results on Pontryagin duality to prove
the following:
 \begin{theorem}\label{t:xaat}
Let $K$ be a global field. Let $L/K$ be a $\ZZ^d_p$ extension with a finite ramification locus and
$d \geq 1$.
Let $A$ be an abelian variety defined over $K$.
Suppose that  $A$ has potentially ordinary reduction at each ramified place of $L/K$. Then the characteristic ideal of the
Pontryagin dual $X_p(A/L)$ of the Selmer group of $A$ and that of
 the Pontryagin dual $X_p(A^t/L)$  of the Selmer group of
  of the dual abelain variety
 $A^t$ satisfy the following equation
$$\chi(X_p(A/L))^\sharp=\chi( X_p(A^t/L))=\chi(X_p(A/L))=\chi( X_p(A^t/L))^\sharp.$$
\end{theorem}

In section \S\ref{s:al}, we prove functional equations for characteristic ideals of Pontryagin duals of the
projections of Selmer groups by central idempotents. This  proves to be a powerful tool for  solving the Iwasawa Main Conjecture in the constant ordinary case (\cite{LLTT}).

Characteristic ideal $\chi$ is explained in $\S$\ref{sss:krul}
 and $X_p(A/L)$ is defined in $\S$\ref{se:se}.

Theorem \ref{t:al} implies the corresponding
equation of characteristic ideals.
 Using Fitting ideals
Mazur and Wiles \cite{MW84} proves such functional equation the $d=1$ case, in which $\fA$ is automatically a $\bf{T}$-system.
For $d\geq 2$, as far as we know,
no progress has been made  and Fitting ideals do not seem to yield a promising approach for a proof.
Our proof is very intricate. Even in the case of
 Cassels-Tate systems, our theorem is not a straightforward consequence of the control theorems in the number field (\cite{gr03}) or function field case (\cite{bl09}, \cite{tan10a}) that one might have expected.
 See in particular our use of an old result of Monsky (Theorem \ref{t:monsky}).

\begin{subsubsection}*{Acknowledgements} The fourth author has been supported by EPSRC. He would like also to express his gratitude to Takeshi Saito for his hospitality at the University of Tokyo where part of this work has been written. Authors 2, 3 and 4 thank Centre de Recerca Matem\`atica for hospitality while working on part of this paper.
Authors 1, 2 and 3 have been partially supported by the National Science Council of Taiwan, grants NSC98-2115-M-110-008-MY2, NSC100-2811-M-002-079 and NSC99-2115-M-002-002-MY3 respectively. Finally, it is our pleasure to thank NCTS/TPE for supporting a number of meetings of the authors in National Taiwan University. \end{subsubsection}

\begin{section}{Preparations}\label{s:setting}

In this section we set up notations for later use.

\begin{subsection}{Iwasawa modules}\label{su:ba}

\subsubsection{}\label{sss:krul}
Let $M$ be a finitely generated $\La$-module. We write $\chi(M)=\chi_{\Gamma}(M)\subset\La$ for
its \textit{characteristic ideal}. Thus, $\chi(M)=0$, if and only if $M$ is non-torsion.
Suppose $M$ is torsion.
By the general theory of modules over a Krull domain, there is a pseudo-isomorphism
\begin{equation} \label{e:structureM} \Phi\colon\bigoplus_{i=1}^m \La/\xi_i^{r_i}\La\longrightarrow M,\end{equation}
where each $\xi_i$ is irreducible. In this situation, we have
$$\chi(M)=\prod_{i=1}^m (\xi_i^{r_i}).$$
It follows that $\chi(M)=\La$, if and only if $M$ is pseudo-null.

\begin{lemma}\label{l:psn}
A finitely generated $\La$-module $M$ is pseudo-null if and only if there exist relatively prime $f_1,...,f_k\in\La$, $k\geq2$, such that $f_iM=0$ for every $i$.
\end{lemma}

\begin{proof}
Since $\La$ is a unique factorization domain, all height one prime ideals are principal and the claim follows.
\end{proof}

Denote $[M]:=\bigoplus_{i=1}^m \La/\xi_i^{r_i}\La$.
Since a non-zero element in $[M]$ cannot be simultaneously annihilated by relatively prime elements of $\La$, there is no non-trivial pseudo-null submodule of $[M]$, and hence $\Phi$ in \eqref{e:structureM} is an embedding. The module $[M]$ is uniquely determined by $M$ up to isomorphism, while $\Phi$ is not. However, we shall fix one such $\Phi$ and view $[M]$ as a submodule of $M$.

\begin{lemma} \label{l:pseudoisom} A composition of pseudo-injections {\em{(}}resp.\! pseudo-surjections,  resp.\! pseudo-isomorphisms{\em{)}} is a pseudo-injection {\em{(}}resp.\! pseudo-surjection,  resp.\! pseudo-isomorphism{\em{)}}. Pseudo-isomorphism is an equivalence relation in the category of finitely generated torsion $\La$-modules.
\end{lemma}

\begin{proof}Let $\alpha\colon M\to N$ and $\beta\colon N\to P$ be two pseudo-injections. Denote by $\ker_\alpha$,
$\ker_\beta$, and $\ker_{\beta\circ \alpha}$ the kernels of $\alpha$, $\beta$ and $\beta\circ \alpha$, respectively.
The exact sequence
$$\xymatrix{0\ar[r] & \ker_\alpha \ar[r] & \ker_{\beta\circ\alpha} \ar[r]^-{\beta} & \ker_\beta \ar@{->>}[r] & \coker}$$
implies $\chi(\ker_{\beta\circ\alpha})\cdot \chi(\coker)=\chi(\ker_\alpha)\cdot \chi(\ker_\beta)$.
Since $\ker_\alpha$, $\ker_\beta$, and hence $\coker$, are pseudo-null, $\chi(\ker_{\beta\circ\alpha})=\La$.
%This proves the pseudo-injection case, o
Other cases can be similarly proved.
\end{proof}

\begin{lemma} \label{l:pseudoinj}
Let $\alpha\colon M\to N$ and $\beta\colon N\to M$ be two pseudo-injections of finitely generated torsion $\La$-modules. Then $\beta\circ\alpha$ is a pseudo-isomorphism.
\end{lemma}

\begin{proof} We have $\chi(\ker_{\beta\circ\alpha})\cdot \chi(M)=\chi(M)\cdot \chi(\coker_{\beta\circ\alpha})$ deduced from
the sequence
$$\xymatrix{0\ar[r] & \ker_{\beta\circ\alpha} \ar[r] & M \ar[r] & M \ar[r] & \coker_{\beta\circ\alpha} \ar[r] & 0}.$$
By Lemma \ref{l:pseudoisom}, $\chi(\ker_{\beta\circ\alpha})=\La$, whence $\chi(\coker_{\beta\circ\alpha})=\La$.
\end{proof}

\end{subsection}

%**************************

\subsection{Twists}
Any continuous group homomorphism $\Gamma\rightarrow\La^\times$ gives rise by linearity to an
endomorphism $\Lambda\rightarrow\Lambda$.
\subsubsection{}\label{ss:iwH1}
An example is the map
${}^\sharp\colon\Lambda\,\lr \Lambda$ we have defined by using
$\gamma \to \gamma^{-1}$ for $\gamma \in \Gamma$.
The particular importance of this map  for us
stems from the fact that if $\langle\;,\;\rangle$ is
a $\Gamma$-invariant pairing between $\La$-modules then
\begin{equation} \label{e:twistpairing} \langle \lambda\cdot a,b\rangle=\langle a,\lambda^\sharp\cdot b\rangle
\end{equation}
for any $\lambda\in\La$.

Suppose $\phi\colon{\Gamma}\rightarrow \ZZ_p^{\times}$ is a continuous homomorphism. Define
$\phi^*\colon\Lambda \to \Lambda$ to
be the ring homomorphism determined by $\phi^*(\gamma):=\phi(\gamma)^{-1}\cdot\gamma$ for $\gamma\in\Gamma$. Since on $\Gamma$ the composition $\phi^*\circ (1/\phi)^*$ is the identity map, we see that $\phi^*$ is an isomorphism on $\Lambda$.

\subsubsection{}\label{su:twistmodule}

Let $M$ be a $\La$-module. Any endomorphism $\alpha\colon\La\to\La$
defines a twisted $\La$-module
$\La \,{}_\alpha\!\otimes_{\La}M\,,$ where the action on the copy of $\La$ on the left is via $\alpha$ (i.e., we have $(\alpha(\lambda)\mu)\otimes m=\mu\otimes\lambda m$ for
$\lambda$, $\mu\in\La$ and $m\in M$) and the module structure is given by
\begin{equation} \label{e:twistaction} \lambda\cdot(\mu\otimes m) := (\lambda\mu)\otimes m \end{equation}
(where $\lambda\mu$ is the product in $\La$ ).
If moreover $\alpha$ is an isomorphism, $\La\,{}_\alpha\!\otimes_{\La}M$ can be identified with $M$ with the $\Lambda(\Gamma')$-action twisted by $\alpha^{-1}$, since in this case \eqref{e:twistaction} becomes
\begin{equation} \label{e:twistaction2} \lambda\cdot(1\otimes m)=1\otimes\alpha^{-1}(\lambda)m \,.\end{equation}

Following the above  we shall write
$$M^\sharp:=\La\,{}_\sharp\!\otimes_{\La}M.$$
Since $\cdot^\sharp$ is an involution, \eqref{e:twistaction2} shows that the action of $\La$ becomes $\lambda\cdot m=\lambda^\sharp m$.
For any finitely generated torsion $\La$-module $M$, we get a decomposition in
\textit{simple part} and \textit{non-simple part}
$$[M]=[M]_{si}\oplus [M]_{ns},$$
in the following way: recalling that $[M]$ is a direct sum of components $\La/\xi_i^{r_i}\La$, we define $[M]_{si}$ as the sum over those $\xi_i$ which are simple and $[M]_{ns}$ as its complement.

It is easy to check that simple elements are irreducible in $\La$ and that
\begin{equation} \label{e:coprimecriter}(f_{\gamma',\zeta'})=(f_{\gamma,\zeta}) \Longleftrightarrow (\gamma')^{\ZZ_p}=\gamma^{\ZZ_p}\text{ and } \zeta'\in\Gal(\QQ_p(\zeta)/\QQ_p)\cdot\zeta\,. \end{equation}
In particular, we have
\begin{equation}\label{e:fs}
(f_{\gamma,\zeta})^{\sharp}=(f_{\gamma^{-1},\zeta})=(f_{\gamma,\zeta}).
\end{equation}
It follows that
\begin{equation}\label{e:fld}
[M]_{si}^{\sharp}=[M]_{si}\,.
\end{equation}

Let $\phi$ be as in \S\ref{ss:iwH1}. Set
\begin{equation} \label{e:twistbyphi}
M(\phi):=\La\,{}_{\phi^*}\!\otimes_{\La}M. \end{equation}
Note that, if we endow $\ZZ_p$ with the trivial action of $\Gamma$, then the $\La$-module
$\ZZ_p(\phi)$ can be viewed as the free rank one $\ZZ_p$-module with the action of $\Gamma$ through multiplication by $\phi$, in the sense that
$$\gamma \cdot a=\phi(\gamma) a \text{ for all } \gamma\in\Gamma, a\in \ZZ_p(\phi)\,.$$
Then for a $\La$-module $M$ we have
$$M(\phi)=\ZZ_p(\phi)\otimes_{\ZZ_p}M,$$
where $\Gamma$ acts by
$$\gamma\cdot(a\otimes x):=(\gamma\cdot a)\otimes (\gamma\cdot x)=\phi(\gamma)\cdot (a\otimes \gamma x)\,.$$

The proof of the following is straight forward, it can be found in \cite{LLTT}.

\begin{lemma}\label{l:phi[]chi} Let $\alpha$ be an automorphism of $\Lambda$. Suppose $M$ is a finitely generated torsion $\La$-module with
$$[M]=\bigoplus_{i=1}^m \La/\xi_i^{r_i}\La\,.$$
Then
$$[\La\,{}_\alpha\!\otimes_{\La}M]=\La \,{}_\alpha\!\otimes_{\La}[M]=\bigoplus_{i=1}^m \La/\alpha(\xi_i)^{r_i}\La,$$
and hence
$$\chi\big(\La \,{}_\alpha\!\otimes_{\La}M\big)=\alpha(\chi(M)).$$
\end{lemma}

\subsection{Some more notation}

The Pontryagin dual of an abelian group $B$ will be denoted $B^\vee$. Since we are going to deal mostly with finite $p$-groups and their inductive and projective limits, we generally won't distinguish between the Pontryagin dual and the set of continuous homomorphisms into the group of roots of unity $\boldsymbol\mu_{p^{\infty}}:=\cup_m\boldsymbol\mu_{p^{m}}$. Note that we shall usually think of $\boldsymbol\mu_{p^{\infty}}$ as a subset of $\bar\QQ_p$ (hence with the discrete topology), so that for a $\La$-module $M$ homomorphisms in $M^\vee$ will often take value in $\bar\QQ_p$.

We shall denote the $\psi$-part of a $G$-module $M$ (for $G$ a group and $\psi\in G^\vee$) by
\begin{equation} \label{e:psipart} M^{(\psi)}:=\{x\in M\;\mid\; g\cdot x=\psi(g)x\;\text{for all}\;g\in G\}. \end{equation}

%\end{subsection}
\end{section}

%SECTION 3++++++++++++++++++++++++++++++++++++++++++++++++++++++++++++++++++++

\begin{section}{Controlled $\Gamma$-systems and the algebraic functional equation} \label{s:con}
Until \S\ref{s:ctsys}, we do not need our group $\Gamma$ to be a Galois group. However, to simplify the notation,
we shall identify $\Gamma$ as $\Gal(L/K)$ for some $L/K$ so that each open subgroup can be written as $\Gal(L/F)$ for some
finite intermediate extension $F$. Denote $K_n=L^{\Gamma_n}$, the $n$th layer, and set $\Gamma^{(n)}:=\Gal(L/K_n)$.

\begin{subsection}{$\Gamma$-systems}\label{su:ga}

\subsubsection{{\em{$\textbf{T}$}} system}\label{su:Ts}

The definition of $\Gamma$-systems can be extended to the notion of a \textit{complete $\Gamma$-system}, for which we stipulate that for each finite intermediate extension $F$ of $L/K$ there are $\Gal(F/K)$-modules $\fa_F$ and $\fb_F$ with a pairing $\langle\;,\;\rangle_F$, and for any pair $F$, $F'$ of finite intermediate extensions with $F\subset F'$, there are $\Gamma$-morphisms $\fr_F^{F'}$ and  $\fk_F^{F'}$ satisfying the obvious analogues of ($\Gamma$-1)-($\Gamma$-4).

We say that $\fA$ is part of a complete $\Gamma$-system $\{\fa_F,\fb_F,\langle\;,\;\rangle_F,\fr_F^{F'},\fk_F^{F'}\}$ if $\fa_n=\fa_{K_n}$, $\fb_n=\fb_{K_n}$, $\fr_n^m=\fr_{K_n}^{K_m}$ and $\fk_n^m=\fk_{K_n}^{K_m}$. Obviously this implies $\fa=\varprojlim_{F}\fa_F$ and $\fb=\varprojlim_{F}\fb_F$.

Assume that $\fA$ is a complete $\Gamma$-system. Let $F$ be a finite intermediate extension and let $L'/F$ be an intermediate $\ZZ_p^{e}$-extension of $L/F$. Write
$$\fa_{L'/F}=\varprojlim_{F\subset F'\subset L'} \fa_{F'},\;\text{and}\;  \fb_{L'/F}=\varprojlim_{F\subset F'\subset L'} \fb_{F'}.$$
They are modules over $\Lambda_{L'/F}:=\ZZ_p[[\Gal(L'/F)]]$. Set the condition
\begin{enumerate}
\item[({\bf T})]
For every finite intermediate extension $F$ and every intermediate $\ZZ_p^{d-1}$-extension $L'/F$ of $L/F$, $\fa_{L'/F}$ and $\fb_{L'/F}$ are finitely generated and torsion over $\La_{L'/F}$.
\end{enumerate}

By a \textbf{T}-system we mean a complete $\Gamma$-system enjoying the property \textbf{T}.

\subsubsection{Morphisms} \label{ss:mor}
We shall always assume that a $\Gamma$-system $\fA=\{\fa_n,\fb_n,\langle\;,\;\rangle_n,\fr_m^n,\fk_m^n\}$ is
\textit{oriented} in the sense that we have fixed an order of the pairs $(\fa_n,\fb_n)$. We define a morphism of  $\Gamma$-systems
$$\fA=\{\fa_n,\fb_n,\langle\;,\;\rangle_n^{\fA},\fr(\fA)_m^n,\fk(\fA)_m^n\} \longrightarrow\fC=\{\fc_n,\fd_n,\langle\;,\;\rangle_n^{\fC},\fr(\fC)_m^n,\fk(\fC)_m^n\}$$
to be a collection of morphisms of $\Gamma$-modules $f_n\colon\fa_n\rightarrow\fc_n$, $g_n\colon\fd_n\rightarrow\fb_n$ commuting with the structure maps and such that $\langle f_n(a),d\rangle_n^{\fC}=\langle a,g_n(d)\rangle_n^{\fA}$ for all $n$.

A pseudo-isomorphism of  $\Gamma$-systems is a morphism $\fA\rightarrow\fC$ such that the induced maps $\fa\rightarrow\fc$, $\fd\rightarrow\fb$ are pseudo-isomorphisms of $\Gamma$-modules.

\begin{example} \label{eg:morf} {\em Given a $\Gamma$-system $\fA=\{\fa_n,\fb_n\}$ and $\lambda\in\La$, let us write
 $\fa_n[\lambda]$ for the $\lambda$-torsion of $\fa_n$,
namely, consisting of those elements in $\fa_n$ killed by $\lambda$.
We can then
 define $\lambda\cdot\fA:=\{\lambda\fa_n,\lambda^\sharp\fb_n\}$  and $\fA[\lambda]:=\{\fa_n[\lambda],\fb_n/\lambda^\sharp\fb_n\}$, with the pairing and the transition maps induced by those of $\fA$. It is easy to check that $\lambda\cdot\fA$ and $\fA[\lambda]$ are $\Gamma$-systems and that the exact sequences $\fa_n[\lambda]\hookrightarrow\fa_n\twoheadrightarrow\lambda\fa_n$  and $\lambda^\sharp\fb_n \hookrightarrow\fb_n\twoheadrightarrow\fb_n/\lambda^\sharp\fb_n$ provide morphisms of oriented $\Gamma$-systems $\fA[\lambda]\rightarrow\fA$ and $\fA\rightarrow\lambda\cdot\fA$.} \end{example}

\subsubsection{Derived systems}\label{su:de}
Let
$\fA=\{\fa_n,\fb_n,\langle\;,\;\rangle_n,\fr_m^n,\fk_m^n\}$
be a $\Gamma$-system.
In the following, we let $\fk_n$ denote the natural map
$$\fa\times \fb\longrightarrow \fa_n\times \fb_n\,.$$

Suppose for each $n$ we are given a $\Gamma$-submodule $\fc_n\subset \fa_n$ such that $\fr_m^n(\fc_m)\subset \fc_n$ and $\fk_m^n(\fc_n)\subset \fc_m$. Using these, we can obtain two derived $\Gamma$-systems from $\fA$. Let $\ff_n\subset \fb_n$ be the annihilator of $\fc_n$, via the duality induced from $\langle\;,\;\rangle_n$, and let $\fd_n:=\fb_n/\ff_n$.
Then we also have $\fr_m^n(\ff_m)\subset \ff_n$ and $\fk_m^n(\ff_n)\subset \ff_m$.
Hence $\fr_m^n$ induces a morphism $\fc_m\times \fd_m\rightarrow \fc_n\times \fd_n$, which, by abuse of notation, we also denote as $\fr_m^n$. Similarly, we have the morphism $\fk_m^n\colon\fc_n\times \fd_n\rightarrow \fc_m\times \fd_m$ and the pairing $\langle\;,\;\rangle_n$ on $\fc_n\times \fd_n$. Let $\fC$ denote the $\Gamma$-system
$$\{\fc_n,\fd_n,\langle\;,\;\rangle_n,\fr_m^n,\fk_m^n\;\mid\; m,n\in \NN, n\geq m\}.$$
We also write $\fe_n:=\fa_n/\fc_n$ and let $\fE$ denote the $\Gamma$-system
$$\{\fe_n,\ff_n,\langle\;,\;\rangle_n,\fr_m^n,\fk_m^n\;\mid\; m,n\in \NN, n\geq m\}.$$
Then we have the sequences
\begin{equation}\label{e:c}
0\longrightarrow \fc \longrightarrow \fa\longrightarrow \fe\longrightarrow 0
\end{equation}
and
\begin{equation}\label{e:e}
0\longrightarrow \ff \longrightarrow \fb\longrightarrow \fd\longrightarrow 0.
\end{equation}
Here $\fc$, $\fd$, $\fe$ and $\ff$ are the obvious projective limits; the systems $\{\fc_n\}$ and $\{\ff_n\}$ satisfy the Mittag-Leffler condition (because all groups are finite), so \eqref{e:c} and \eqref{e:e} are exact.

\begin{lemma} \label{l:e0f0} Assume $\fa_n=\fk_n(\fa)$ for all $n$. Then $\fe\sim 0$ implies $\ff\sim 0$.
\end{lemma}

\begin{proof} The assumption implies $\fk_n(\fe)=\fe_n$. Thus $f\cdot \fe=0$ implies $f\cdot \fe_n=0$, and consequently, by the duality, $f^{\sharp}\cdot \ff_n=0$ for all $n$, yielding $f^{\sharp}\cdot \ff=0$. Now apply Lemma \ref{l:psn}.
\end{proof}

\subsubsection{The system $\fA'$} \label{ss:a'}
In the following case, we apply the above two methods together. We first get a system $\{\fa_n/\fa^0_n,\fb^1_n\}$ by putting
\begin{equation} \label{e:dirlim0} \fa_n^0\times \fb_n^0:=\bigcup_{n'\geq n}\Ker(\fr_n^{n'})= \Ker\!\big(\fa_n\times \fb_n \longrightarrow \lim_{\stackrel{\longrightarrow}{m}}\fa_m\times \fb_m\big) \end{equation}
and letting $\fa_n^1$, $\fb_n^1$ be respectively the annihilators of $\fb_n^0$, $\fa_n^0$, via $\langle\;,\;\rangle_n$.
Then we apply the $\fC$-construction to $\{\fa_n/\fa^0_n,\fb^1_n\}$ defining $\fa_n'\subset\fa_n/\fa_n^0$ via
$$\fa_n'\times \fb_n':=\image\!\big(\fa_n^1\times \fb_n^1\longrightarrow (\fa_n/\fa_n^0)\times (\fb_n/\fb_n^0)\big)\,.$$
Notice that $\fb_n'$ is dual to $\fa_n'$, as can be seen by dualizing the diagram
\[\begin{CD}
0 @>>> \fa^1_n @>>> \fa_n \\
&& @VVV @VVV \\
0 @>>> \fa_n' @>>> \fa_n/\fa^0_n  \end{CD}\]
(recall that the duals of $\fa^1_n$ and $\fa_n/\fa^0_n$ are respectively $\fb_n/\fb_n^0$ and $\fb^1_n$).\\
Thus we get a $\Gamma$-system
$$\fA':=\{\fa_n',\fb_n',\langle\;,\;\rangle_n,\fr_m^n,\fk_m^n\;\mid\; m,n\in \NN, n\geq m\}.$$

Denote, for $i=0,1$,
$$\fa^i\times \fb^i:=\lim_{\stackrel{\leftarrow}{n}} \fa_n^i\times \fb_n^i,$$
and
$$\fa'\times \fb':=\lim_{\stackrel{\leftarrow}{n}}\fa_n'\times \fb_n'=\image\!\big(\fa^1\times\fb^1\longrightarrow (\fa/\fa^0)\times (\fb/\fb^0)\big).$$
The pairings $\langle\;,\;\rangle_n$ allow identifying each $\fa_n\times\fb_n$ with its own Pontryagin dual and this identification is compatible with the maps $\fr_m^n$, $\fk_m^n$. Then $\fa\times\fb$ is the dual of $\displaystyle \lim_{\rightarrow} \fa_n\times\fb_n$. Consider the exact sequence
\begin{equation}   \begin{CD} \label{e:exseqX}
0 @>>> \fa_n^0\times\fb^0_n @>>> \fa_n\times\fb_n @>>> (\fa_n\times\fb_n)/(\fa_n^0\times\fb^0_n) @>>> 0. \end{CD}\end{equation}
By construction, $\fa_n^1\times\fb^1_n$ is the dual of $(\fa_n\times\fb_n)/(\fa_n^0\times\fb^0_n)$. The inductive limit of \eqref{e:exseqX} gets the identity
$$\lim_{\rightarrow} \fa_n\times\fb_n=\lim_{\rightarrow}\, (\fa_n\times\fb_n)/(\fa_n^0\times\fb^0_n)$$
($\varinjlim\fa_n^0\times\fb_n^0=0$ is immediate from \eqref{e:dirlim0}) and hence, taking duals,
$$\fa^1\times \fb^1=\fa\times \fb.$$
Thus we have an exact sequence
\begin{equation}\label{e:0'}
0\longrightarrow \fa^0\times \fb^0\longrightarrow \fa\times \fb\longrightarrow \fa'\times \fb'\longrightarrow 0.
\end{equation}

\subsubsection{Strongly-controlled $\Gamma$-systems}\label{su:st}
In the previous section we saw that, since $\fb=\fb^1$ and $\fa=\fa^1$ , the information carried by $\fa^0$ and $\fb^0$
does not pass to $\fb$ and $\fa$: this explains Definition \ref{d:scontr}. Here we consider a condition stronger than being pseudo-controlled.

\begin{definition}\label{d:scontr}
A $\Gamma$-system $\fA$ is strongly controlled if  $\fa_n^0\times \fb_n^0=0$ for every $n$.
\end{definition}

\begin{lemma}\label{l:eq}
A $\Gamma$-system $\fA$ is strongly controlled if and only if $\fr_m^n$ is injective  {\em{(}}resp. $\fk_m^n$ is surjective{\em{)}} for $n\geq m$.
\end{lemma}

\begin{proof} The definition and the duality. \end{proof}

\begin{lemma}\label{l:st}
Suppose $\fA$ is a $\Gamma$-system. Then the following holds:
\begin{enumerate}
\item the system $\fA'$ is strongly controlled;
\item if $\fA$ is pseudo-controlled, then $\fa\sim \fa'$ and $\fb\sim \fb'$.
\end{enumerate}
\end{lemma}

\begin{proof} Statement (1) follows from the definition of $\fA'$ and (2) is immediate from the exact sequence \eqref{e:0'}. \end{proof}

%\begin{lemma} \label{l:pcsuffic} Let $\fA$ be a pseudo-controlled $\Gamma$-system.
%Then the functional equations \eqref{e:basharp} and \eqref{e:chisharp} hold for $\fA$ if and only if they hold for $\fA'$. \end{lemma}

%\begin{proof} Obvious by Lemma \ref{l:st}(2). \end{proof}

\begin{lemma}\label{l:an} Suppose $\fA$ is strongly controlled. Then $\xi\cdot \fb=0$, for some $\xi\in\La$, if and only if $\xi^{\sharp}\cdot \fa=0$. \end{lemma}

\begin{proof} By Lemma \ref{l:eq}, we have $\fb_n=\fk_n(\fb)$.  Thus $\xi\cdot \fb=0$ implies $\xi\cdot \fb_n=0$, and consequently, by the duality, $\xi^\sharp\cdot \fa_n=0$ for all $n$, yielding $\xi^\sharp\cdot \fa=0$. \end{proof}

\end{subsection}

%===================================================================

\subsection{Two maps} In this subsection we introduce the maps $\Phi$ and $\Psi$ which
plays a key role in our constructions.

For simplicity, in the following we shall use the notations $Q_n:=\QQ_p[\Gamma_n]$ and $\Lambda_n:=\ZZ_p[\Gamma_n]$. The projections $\pi_m^n\colon\Gamma_n\to\Gamma_m$ are canonically extended to ring morphisms $\colon\Lambda_n\to\Lambda_m$. Let
$$Q_\infty:=\varprojlim Q_n=\QQ_p[[\Gamma]]\,.$$
Thanks to the inclusions $\La_n\hookrightarrow Q_n$ we can see $\La$ as a subring of $Q_\infty$.

\subsubsection{The Fourier map} \label{ss:four} Let $\fA$ be a $\Gamma$-system as above. In this section, we construct a $\Lambda$-linear map
$$\Phi\colon\fa^\sharp\lr \Hom_{\Lambda}(\fb,Q_\infty/\La)\,.$$
First recall that the pairing in ($\Gamma$-4) induces for any $n$ an isomorphism of $\La$-modules
%\footnote{Here $\Hom_{\ZZ_p}(\fb_n,\QQ_p/\ZZ_p)$ is a $\La$-module via $(\gamma\cdot f)\colon x\mapsto f(\gamma x)$ for $\gamma\in\Gamma$.}
$$\fa_n^\sharp\simeq \Hom_{\ZZ_p}(\fb_n,\QQ_p/\ZZ_p),$$
the twist by the involution $\cdot^\sharp$ being due to \eqref{e:twistpairing}. Equality \eqref{e:rkpairing} shows that these isomorphisms form an isomorphism of projective systems, where the right hand side is endowed  with the transition maps induced by the direct system $(\fb_n,\fr_m^n)$.
Passing to the projective limit, we deduce a $\Lambda$-isomorphism
$$\fa^\sharp\simeq \lim_{\stackrel{\leftarrow}{n}}\Hom_{\ZZ_p}(\fb_n,\QQ_p/\ZZ_p).$$
Now the map $\Phi$ is obtained as the composed of this isomorphism and the following  $\Lambda$-linear maps:
\begin{enumerate}
\item[($\Phi$-1)] the homomorphism
$$\lim_{\stackrel{\leftarrow}{n}}\Hom_{\ZZ_p}(\fb_n,\QQ_p/\ZZ_p) \lr \lim_{\stackrel{\leftarrow}{n}} \Hom_\Lambda(\fb_n,Q_n/\La_n)$$
obtained by sending $(f_n)_n$ to $\big(\hat{f}_n\colon x\mapsto\sum_{\gamma\in\Gamma_n}f_n(\gamma^{-1}x)\gamma\big)_n\,$;
\item[($\Phi$-2)] the homomorphism
$$\varprojlim \Hom_\Lambda(\fb_n,Q_n/\La_n)\lr \varprojlim \Hom_{\Lambda}(\fb,Q_n/\La_n)$$
induced by $\fk_n\colon\fb\to \fb_n\,$;
\item[($\Phi$-3)] the canonical isomorphism
$$\lim_{\stackrel{\leftarrow}{n}} \Hom_{\Lambda}(\fb,Q_n/\La_n)\simeq  \Hom_{\Lambda}(\fb,\lim_{\stackrel{\leftarrow}{n}}Q_n/\La_n)$$
and the identification $\varprojlim Q_n/\La_n=Q_\infty/\La$ (since the maps $\La_n\rightarrow\La_m$ are surjective).
\end{enumerate}
Here as transition maps in $\varprojlim\Hom_\Lambda(\fb_n,Q_n/\La_n)$ we take (for $n\geq m$)
$$\Hom_\Lambda(\fb_n,Q_n/\La_n)\lr\Hom_\Lambda(\fb_m,Q_m/\La_m)$$
\begin{equation}\label{e:twistmap}\varphi \mapsto p^{-d(n-m)}(\pi_m^n\circ\varphi\circ\fr_m^n)\,.\end{equation}
We have to check that ($\Phi$-1) and ($\Phi$-2) define maps of projective systems. For ($\Phi$-1),
this means to verify that for any $n\geq m$ we have \begin{equation}\label{e:check1}\hat{f_m}=p^{-d(n-m)}(\pi_m^n\circ\hat{f_n}\circ\fr_m^n)\,, \end{equation}
where, by definition, $f_m=f_n\circ\fr_m^n$. For $x\in\fb_m$,
$$\pi_m^n(\hat{f_n}(\fr_m^nx))=\pi_m^n\big(\sum_{\gamma\in\Gamma_n}f_n(\gamma^{-1}(\fr_m^n x))\gamma\big)=\sum_{\gamma\in\Gamma_n}f_n(\gamma^{-1}\fr_m^n x)\pi_m^n(\gamma)$$
(using the fact that, by ($\Gamma$-2), $\fr^n_m$ is a $\Gamma$-morphism)
$$=\sum_{\gamma\in\Gamma_n}f_n(\fr_m^n(\gamma^{-1}x))\pi_m^n(\gamma)=\frac{|\Gamma_n|}{|\Gamma_m|}\sum_{\gamma\in\Gamma_m}f_n(\fr_m^n(\gamma^{-1}x))\gamma=p^{d(n-m)}\hat{f}_m(x)\,,$$
so \eqref{e:check1} holds. As for ($\Phi$-2), the transition map
$$\Hom_{\Lambda}(\fb,Q_n/\La_n)\lr \Hom_{\Lambda}(\fb,Q_m/\La_m)$$
is $\psi\mapsto\pi^n_m\circ\psi$ and the map defined in ($\Phi$-2) is $(\varphi_n)_n\mapsto (\varphi_n\circ\fk_n)_n\,$. By \eqref{e:twistmap},
$$\varphi_m\circ\fk_m=p^{-d(n-m)}(\pi_m^n\circ\varphi_n\circ\fr_m^n)\circ\fk_m=p^{-d(n-m)}(\pi_m^n\circ\varphi_n\circ\fr_m^n\circ\fk^n_m\circ\fk_n)=$$
(by property ($\Gamma$-3) of $\Gamma$-systems)
$$=p^{-d(n-m)}(\pi_m^n\circ\varphi_n\circ\Nm_{\Gamma_n/\Gamma_m}\circ \fk_n)=\pi_m^n\circ\varphi_n\circ\fk_n$$
(since $\varphi_n$, being a $\La$-morphism, commutes with $\Nm_{\Gamma_n/\Gamma_m}$ and $\pi_m^n\circ\Nm_{\Gamma_n/\Gamma_m}=p^{d(n-m)}\pi_m^n$).
So also ($\Phi$-2) is a map of projective systems.

\begin{remark}{\em Actually, one can also check that the maps $f_n\mapsto\hat{f_n}$ used in ($\Phi$-1) are isomorphisms. The inverse is $f\mapsto\delta_e\circ f$, where $\delta_e\colon Q_n/\La_n\to\QQ_p/\ZZ_p$ is the function sending $\sum_{\Gamma_n}a_\gamma\gamma$ to $a_e$ ($e$ being the neutral element in $\Gamma_n$).}
\end{remark}

If the $\Gamma$-system $\fA$ is strongly controlled then the map $\Phi$ is clearly injective (since $\fb$ maps onto $\fb_n$ for all $n$). In general, we have the following.

\begin{lemma} \label{l:kernPhi}
The kernel of $\Phi$ equals $(\fa^0)^\sharp$.
\end{lemma}

\begin{proof} The image of $a=(a_n)_n\in\fa$ in $\varprojlim\Hom_\Lambda(\fb_n,Q_n/\La_n)$ is the map
$$b=(b_n)_n\mapsto\big(\sum_{\gamma\in\Gamma_n}\langle a_n,\gamma^{-1}b_n\rangle_n\gamma\big)_n\,.$$
To conclude, observe that $\fa_n\to\varinjlim\fa_m$ is dual to $\fb\to\fb_n$. Hence $\langle a_n,b_n\rangle_n=0$, for every $b_n$ contained in the image of $\fb\rightarrow \fb_n$, if and only if $a_n\in\fa_n^0$.
\end{proof}

\subsubsection{} Let $\fb$ be a finitely generated torsion $\La$-module. In \S\ref{ss:Psi} below we shall construct a map
$$\Psi\colon\Hom_{\Lambda}(\fb, Q_\infty/\La)\to \Hom_{\Lambda}(\fb,Q(\Lambda)/\Lambda),$$
where $Q(\La)$ is the field of fractions of $\La$. The interest of having such a $\Psi$ comes from the following lemma.

\begin{lemma} \label{l:b&Iwadj}
For $\fb$ a finitely generated torsion $\La$-module, we have a pseudo-isomorphism
$$\fb\sim \Hom_{\Lambda}(\fb,Q(\Lambda)/\Lambda).$$
\end{lemma}

\begin{proof}
From the exact sequence
$$0\longrightarrow [\fb]\longrightarrow \fb\longrightarrow \fn\lr 0,$$
where $\fn$ is pseudo-null, we deduce the exact sequence
$$\Hom_{\Lambda}(\fn,Q(\La)/\Lambda)\hookrightarrow \Hom_{\Lambda}(\fb,Q(\La)/\La) \to \Hom_{\Lambda}([\fb],Q(\La)/\La)\to\text{Ext}_{\Lambda}^1(\fn,Q(\La)/\La)\,.$$
The annihilator of $\fn$ also kills $\Hom_{\Lambda}(\fn,\;)$ and its derived functors, so by Lemma \ref{l:psn}, it follows $\Hom_{\La}(\fb,Q(\Lambda)/\Lambda)\sim \Hom_{\La}([\fb],Q(\Lambda)/\Lambda)$, and we can assume that $\fb=[\fb]$. Write
$$\fb=\Lambda/(\xi_1)\oplus\dots\oplus\Lambda/(\xi_n)\,.$$
Then
$$ \Hom_{\Lambda}(\fb,Q(\Lambda)/\Lambda)=\bigoplus\Hom_{\Lambda}(\La/(\xi_i),Q(\Lambda)/\Lambda)=\bigoplus\Hom_{\Lambda}(\Lambda/(\xi_i),\xi_i^{-1}\Lambda/\Lambda)$$
because $(Q(\La)/\La)[\xi_i]=\xi_i^{-1}\La/\La.$
Since
$$\Hom_{\Lambda}(\Lambda/(\xi_i),\xi_i^{-1}\Lambda/\Lambda)=\Lambda/(\xi_i)\,,$$
we conclude that in this situation, $\Hom_{\Lambda}(\fb,Q(\Lambda)/\Lambda)=\fb\,.$
\end{proof}

\subsubsection{A theorem of Monsky} \label{su:monsky}
Let $\Gamma^\vee$ (resp. $\Gamma^\vee_n$ ) denote the group of continuous characters $\Gamma\rightarrow \boldsymbol\mu_{p^\infty}$ (resp. $\Gamma_n\rightarrow \boldsymbol\mu_{p^{\infty}}$); we view $\Gamma^\vee_n$ as a subgroup of $\Gamma^\vee$. For each $\omega\in \Gamma^\vee$, let $E_\omega:=\QQ_p(\boldsymbol\mu_{p^m})\subset \bar\QQ_p$ be the subfield generated by the image $\omega(\Gamma) =\boldsymbol\mu_{p^m}$, and write $\cO_\omega=\ZZ_p[\boldsymbol\mu_{p^m}]$. Then $ \omega$ induces a continuous ring homomorphism $\omega\colon\Lambda\rightarrow \cO_\omega\subset E_\omega$.\\

Let $\xi\in\La$: we say that $\omega$ is a zero of $\xi$, if and only if $\omega(\xi)=0$, and denote the zero set
\begin{equation} \label{e:Delta}
\triangle_{\xi}:=\{ \omega\in {\Gamma^\vee}\;\mid\; \omega(\xi)=0\}.
\end{equation}
Then we recall a theorem of Monsky (\cite[Lemma 1.5 and Theorem 2.6]{monsky}).

\begin{definition} \label{d:monsky}
A subset $\Xi\subset \Gamma^\vee$ is called a $\ZZ_p$-flat of codimension $k$, if there exists $\{\gamma_1,...,\gamma_k\}\subset \Gamma$ expandable to a $\ZZ_p$-basis of $\Gamma$ and $\zeta_1,...,\zeta_k\in\boldsymbol\mu_{p^{\infty}}$ such that
$$\Xi=\{ \omega \in \Gamma^\vee\;\mid\; { \omega}(\gamma_i)=\zeta_i,i=1,...,k\}.$$
\end{definition}

\noindent This definition is due to Monsky: in \cite[\S1]{monsky}, he proves that $\ZZ_p$-flats generate the closed sets of a certain (Noetherian) topology on $\Gamma^\vee$. It turns out that in this topology the sets $\triangle_\xi$ are closed, and they are proper subsets (possibly empty) if $\xi\neq0$ (\cite[Theorem 2.6]{monsky}). Hence

\begin{mytheorem}[Monsky] \label{t:monsky}
Suppose $\cO$ is a discrete valuation ring finite over $\ZZ_p$ and $\xi\in\cO[[\Gamma]]$ is non-zero. Then the zero set $\triangle_{\xi}$ is a proper subset of $\Gamma^\vee$ and is a finite union of $\ZZ_p$-flats.
\end{mytheorem}

\subsubsection{Structure of $Q_\infty$} \label{ss:qorbit}
The group $\Gal(\bar\QQ_p/\QQ_p)$ acts on $\Gamma^\vee$ by $(\sigma\cdot\omega)(\gamma):=\sigma(\omega(\gamma))$. Let $[\omega]$ denote the $\Gal(\bar\QQ_p/\QQ_p)$-orbit of $ \omega$. Attached to any character $ \omega\in\Gamma^\vee_n$ there is an idempotent
\begin{equation} \label{e:idempotent} e_\omega:=\frac{1}{|\Gamma_n|}\sum_{\gamma\in\Gamma_n} \omega(\gamma^{-1})\gamma\in\bar\QQ_p[\Gamma_n]\,. \end{equation}
Accordingly, we get the decomposition
$$\QQ_p[\Gamma_n]=\bar\QQ_p[\Gamma_n]^{\Gal(\bar\QQ_p/\QQ_p)}=\big(\prod_{\omega\in\Gamma^\vee_n}e_\omega\bar\QQ_p[\Gamma_n]\big)^{\Gal(\bar\QQ_p/\QQ_p)}=\prod_{[ \omega]\subset \Gamma^\vee_n}  E_{[\omega]},$$
where $[\omega]$ runs through all the ${\Gal(\bar\QQ_p/\QQ_p)}$-orbits of $\Gamma^\vee_n$ and
$$E_{[\omega]}:=(\prod_{\chi\in [\omega]}e_\chi\bar\QQ_p[\Gamma_n])^{\Gal(\bar\QQ_p/\QQ_p)}.$$
Observe that the homomorphism $\omega\colon\QQ_p[\Gamma_n]\rightarrow\bar\QQ_p$ induces an isomorphism $E_{[\omega]}\simeq E_\omega$ (the inverse being given by $1\mapsto\sum_{\sigma\in Gal(E_\omega/\QQ_p)}\sigma(e_\omega)=(e_\chi)_{\chi\in [\omega]}$).

Since $\pi^n_m(e_\omega)$ equals $e_{\omega'}$ if $\omega=\omega'\circ\pi^n_m$ and is $0$ otherwise, we have the commutative diagram
\[\begin{CD}  \QQ_p[\Gamma_n] @>>> \prod_{[ \omega]\subset \Gamma^\vee_n}  E_{[\omega]}\\
   @VV{\pi_m^n}V @VVV \\
\QQ_p[\Gamma_m] @>>> \prod_{[ \omega]\subset \Gamma^\vee_m}  E_{[\omega]}  \end{CD}\]
where the right vertical arrow is the natural projection by the inclusion $\Gamma^\vee_n\hookrightarrow\Gamma^\vee_m$.
It follows that we have identities
\begin{equation} \label{e:decQinfty} Q_\infty=\lim_{\stackrel{\leftarrow}{n}}Q_n\simeq\prod_{[\omega]\subset\Gamma^\vee}E_{[\omega]} \end{equation}
so that
\begin{equation} \label{e:torsQinfty} Q_\infty[\lambda]=\prod_{[\omega]\subset \triangle_\lambda} E_{[\omega]}
\end{equation}
for all $\lambda\in\La$ (here $Q_\infty[\lambda]$ denotes the $\lambda$-torsion subgroup).

\subsubsection{The map $\Psi$} \label{ss:Psi}
Let $\fb$ be a finitely generated torsion $\La$-module: so is $ \Hom_{\Lambda}(\fb,Q_\infty/\La)$.
We assume that $\xi\cdot \fb=0$, for some non-zero $\xi\in\Lambda$. Let $\triangle_\xi^c:=\Gamma^\vee-\triangle_\xi$ denote the complement of $\triangle_\xi$. From \eqref{e:decQinfty} and \eqref{e:torsQinfty} one deduces the direct sum decomposition
$$Q_\infty=Q_\infty[\xi]\oplus Q_\infty^c\,,$$
where $Q_\infty^c=\prod_{[\omega]\subset \triangle_\xi^c} E_{[\omega]}$.
Let $\varpi\colon Q_\infty\rightarrow Q_\infty^c$ be the natural projection and put $\La^c:=\varpi(\La)$ (here $\La$ is thought of as a subset of $Q_\infty$ via the maps $\ZZ_p[\Gamma_n]\hookrightarrow\QQ_p[\Gamma_n]$).

\begin{lemma} \label{l:xitorsionHom} We have a $\Lambda$-isomorphism
$$\Hom_{\La}(\fb,Q_\infty^c/\La^c) \simeq \Hom_{\Lambda}(\fb,Q(\Lambda)/\Lambda).$$
\end{lemma}

\begin{proof}
Since $\fb$ is annihilated by $\xi$, the image of each $\eta\in\Hom_{\La}(\fb,Q_\infty^c/\La^c)$ is contained in $(Q_\infty^c/\La^c)[\xi]\,.$
Note that, since $\omega(\xi)\not=0$ for every $\omega\in\triangle_\xi^c$, the element $\varpi(\xi)$ is a unit in $Q_\infty^c$. Denote
$$\xi^{-1}\La^c:=\{x\in Q_\infty^c\;\mid\; \xi\cdot x\in\La^c\}.$$
Then
$$(Q_\infty^c/\La^c)[\xi]=\xi^{-1}\La^c/\La^c$$
and hence
$$\Hom_{\La}(\fb,Q_\infty^c/\La^c)=\Hom_{\La}(\fb,\xi^{-1}\La^c/\La^c).$$
Similarly,
$$\Hom_{\La}(\fb,Q(\La)/\La)=\Hom_{\La}(\fb,\xi^{-1}\La/\La).$$
To conclude the proof, it suffices to show that $\varpi\colon\La\rightarrow \La^c$ is an isomorphism, because then so is the induced map
$$\xi^{-1}\Lambda/\Lambda\,\lr \xi^{-1}\La^c/\La^c.$$
Since $\La^c=\varpi(\La)$ by definition, we just need to check injectivity. Suppose $\varpi(\epsilon)=0$ for some $\epsilon\in\Lambda$. Then $\omega(\epsilon)=0$ for every $\omega\not\in\triangle_{\xi}$, and hence $\omega(\xi\epsilon)=0$ for every $ \omega\in\Gamma^\vee$. Monsky's theorem (or, alternatively, the isomorphism \eqref{e:decQinfty}) implies that $\xi\epsilon=0$ and hence $\epsilon=0$.
\end{proof}

Let
\begin{equation} \label{e:upsilondef} \Upsilon\colon\Hom_{\La}(\fb,Q_\infty/\La)\,\lr\Hom_{\La}(\fb,Q_\infty^c/\La^c)\end{equation}
be the morphism induced from $\varpi$. By composition of the isomorphism of Lemma \ref{l:xitorsionHom} with $\Upsilon$, we deduce the $\Lambda$-morphism
$$\Psi\colon\Hom_{\La}(\fb,Q_\infty/\La)\,\lr\Hom_{\Lambda}(\fb,Q(\Lambda)/\Lambda).$$

%===================================================================

\begin{subsection}{Proof of the algebraic functional equation} \label{se:al}
In this section, we complete the proof of Theorem \ref{t:al} by proving each of (1), (2), (3), separately. To prove (3) we use (2) and (3) is used in the proof of Theorem \ref{p:idemsha}.

\subsubsection{Non-simple annihilator} \label{ss:simple}
We start the proof of case (1) of Theorem \ref{t:al} by reformulating our hypothesis
\vskip5pt
\noindent
({\bf NS}): {\em{$\xi$ is not divisible by any simple element}}.

\begin{lemma} \label{l:hypH}
Hypothesis {\em (}{\bf NS}{\em)} holds if and only if $\triangle_{\xi}$ contains no codimension one $\ZZ_p$-flat.
\end{lemma}

\begin{proof}
If $\xi$ is divisible by a simple element $f=f_{\gamma,\zeta}$, then
$\triangle_{\xi}$ contains $\triangle_{f}$ which is a union of the codimension one $\ZZ_p$-flats
$$\{ \omega \in \Gamma^\vee\;\mid\; \omega(\gamma)=\sigma(\zeta)\},\;\; \sigma\in\Gal(\QQ_p(\zeta)/\QQ_p).$$
Conversely, assume that $\triangle_{\xi}$ contains the codimension one $\ZZ_p$-flat
$$\Xi=\{ \omega \in \Gamma^\vee\;\mid\; { \omega}(\gamma)=\zeta\}.$$
Each $\omega\in\Xi$ factors through
$$\pi\colon\Lambda \longrightarrow \ZZ_p[\zeta][[\Gamma]]/(\gamma-\zeta)=\ZZ_p[\zeta][[\Gamma']]\,,$$
where $\Gamma'$ is the quotient $\Gamma/\gamma^{\ZZ_p}$, and vice versa every continuous character of $\Gamma'$ can be uniquely lifted to a character in $\Xi$. Thus the zero set of $\pi(\xi)\in\ZZ_p[\zeta][[\Gamma']]$ equals $(\Gamma')^\vee$. Then Monsky's theorem implies that $\pi(\xi)=0$ and hence is divisible by $\gamma-\zeta$ in $\ZZ_p[\zeta][[\Gamma]]$.
This implies that $\xi$ is divisible by $f_{\gamma,\zeta}$ in $\Lambda$.
\end{proof}

By Monsky's theorem, we have either $\triangle_{\xi}=\emptyset$ or $\triangle_{\xi}=\cup_j\,\Xi_j\,$, with
$$\Xi_j=\{ \omega \in \Gamma^\vee\;\mid\; { \omega}(\gamma_i^{(j)})=\zeta_i^{(j)},i=1,...,k^{(j)}\}. $$
In the second case, for all $j$ let $G_j$ be the $\ZZ_p$-submodule of $\Gamma$ generated by the $\gamma_i^{(j)}$'s, $i=1,...,k^{(j)}$: if ({\bf NS}) holds, each $G_j$ has rank at least 2. Hence, since there is just a finite number of $j$, it is possible to choose $\{\sigma_1^{(j)},\sigma_2^{(j)}\}_j$ such that $\sigma_i^{(j)}\in G_j-\Gamma^p$ and each pair $(\sigma_i^{(j)},\sigma_{i'}^{(j')})$ consists of $\ZZ_p$-independent elements unless $(i,j)=(i',j')$. Let $\varepsilon_i^{(j)}$ denote the common value that all characters in $\Xi_j$ take on $\sigma_i^{(j)}$ and write
\begin{equation} \label{e:phi}
\varphi_1:=\prod_j f_{\sigma_1^{(j)},\varepsilon_1^{(j)}}, \qquad
\varphi_2:=\prod_j f_{\sigma_2^{(j)},\varepsilon_2^{(j)}}.\end{equation}
 Then the coprimality criterion \eqref{e:coprimecriter} ensures that $\varphi_1$ and $\varphi_2$ are relatively prime. Moreover $\omega(\varphi_i)=0$ for all $\omega\in\triangle_{\xi}$, that is $\triangle_{\xi}\subseteq\triangle_{\varphi_i}$. By \eqref{e:torsQinfty} it follows
\begin{equation} \label{e:triangles} \varphi_i\cdot Q_\infty[\xi] =0  \text{ for both } i. \end{equation}

\begin{remark}{\em The case $\Delta_\xi\neq\emptyset$ can actually occur. For example, let $\gamma_1,\gamma_2$ be two distinct elements of a $\ZZ_p$-basis of $\Gamma$ and consider
$$\xi=\gamma_1-1+p(\gamma_2-1)+p^2(\gamma_1-1)(\gamma_2-1).$$
Then $\Delta_\xi=\{ \omega \mid \omega(\gamma_1)=\omega(\gamma_2)=1\}$ as one easily sees comparing $p$-adic valuations of the three summands $\omega(\gamma_1-1)$, $\omega(p(\gamma_2-1))$ and $\omega(p^2(\gamma_1-1)(\gamma_2-1))$. }
\end{remark}

\begin{lemma} \label{l:Upsilonpsn}
Assume $\xi\fb=0$ for some $\xi\in\La$ satisfying hypothesis {\em (}{\bf NS}{\em)}. Then the restriction of $\Psi$ to $\Phi(\fa^\sharp)$ is pseudo-injective.
\end{lemma}

\begin{proof} We just need to control the kernel of the map $\Upsilon$ of \eqref{e:upsilondef}. If $\Delta_\xi$ is empty then $\Upsilon$ is the identity and we are done. If not, we show that kernel and cokernel of $\Upsilon$ are annihilated by both $\varphi_1$ and $\varphi_2$ (see \eqref{e:phi}). Consider the exact sequence
$$0\lr \Ker(\varpi) \lr Q_\infty/\La \lr Q_\infty^c/\La^c \lr 0$$
(where by abuse of notation we denote the map induced by $\varpi$ with the same symbol). This induces the exact sequence
$$\Hom_{\Lambda}(\fb,\Ker(\varpi)) \hookrightarrow \Hom_{\Lambda}(\fb,Q_\infty/\La) \to \Hom_{\Lambda}(\fb,Q_\infty^c/\La^c) \to\text{Ext}^1_{\Lambda}(\fb,\Ker(\varpi)).$$
Since $\Ker(\varpi)$ is a quotient of $Q_\infty[\xi]$, \eqref{e:triangles} yields $\varphi_i\cdot\Ker(\varpi)=0$. Therefore both $\Ker(\Upsilon)$ and $\Coker(\Upsilon)$ are annihilated by $\varphi_1$ and $\varphi_2$. (Note that we cannot say that $\Upsilon$ is a pseudo-isomorphism, because $\Hom_{\La}(\fb,Q_\infty/\La)$ is not a finitely generated $\La$-module: e.g., any group homomorphism $\fb\mapsto E_\omega$ for $\omega\in\Delta_\xi$ is also a $\La$-homomorphism.)
\end{proof}

Now we can complete the proof of Theorem \ref{t:al}.(1).

\begin{proof}[{\bf Proof of Theorem \ref{t:al}(1)}] To start with, assume $\xi\fb=0$. Then, by Lemmata \ref{l:kernPhi}, \ref{l:Upsilonpsn} and \ref{l:b&Iwadj}, we get a pseudo-injection $\fa^\sharp\rightarrow \fb$. Moreover, thanks to Lemma \ref{l:st}, we may assume that $\fA$ is strongly controlled. By Lemma \ref{l:an} this implies that $\fa$ is killed by $\xi^\sharp$, which is also not divisible by simple elements. Exchanging the role of $\fa$ and $\fb$, we deduce a pseudo-injection  $\fb^\sharp\rightarrow \fa$ and therefore a pseudo-injection $\fb\rightarrow \fa^\sharp$. The Theorem now follows from Lemma \ref{l:pseudoinj}.

In the general case when $\xi\fb$ is pseudo-null but not 0, we can still assume that $\fA$ is strongly controlled. Let $\ff_n$ be the kernel of the morphism $\fb_n\rightarrow\fb_n$, $b\mapsto\xi b$, and construct two derived systems as in \S\ref{su:de} (but with $\ff_n$ playing the role of $\fc_n$). We get again the two exact sequences \eqref{e:c} and \eqref{e:e}. By hypothesis $\fd=\xi\fb\sim0$ and then Lemma \ref{l:e0f0} implies $\fc\sim0$. Hence
$$\fb\sim\ff\sim\fe^\sharp\sim\fa^\sharp$$
(where the central pseudo-isomorphism holds because $\xi\ff=0$).
\end{proof}

\subsubsection{The non-simple part}\label{ss:nons} Let $\fA'$ be the derived system in \S\ref{ss:a'}.
\begin{corollary}\label{c:nons}
For any $\Gamma$-system $\fA$, we have
$$[\fa']_{ns}^\sharp= [\fb']_{ns}.$$
\end{corollary}

\begin{proof}
By Lemma \ref{l:st}(1) we can lighten notation and assume that $\fA$ is strongly controlled (replacing $\fA$ by $\fA'$ if necessary).
Write $\chi(\fb)=(\lambda\mu)$, with $\chi([\fb]_{ns})=(\lambda)$ and $\chi([\fb]_{si})=(\mu)$. Since $\fb/[\fb]$ is pseudo-null, there are
$\eta_1,\eta_2\in \La$, coprime to each other and both coprime to $\chi(\fb)$, such that $\eta_1\cdot (\fb/ [\fb])=\eta_2\cdot (\fb/ [\fb])=0$ .
Then $\lambda\mu\eta_1\cdot \fb=\lambda\mu\eta_2\cdot \fb=0$. By Lemma \ref{l:an}
\begin{equation}\label{e:918}
(\lambda\mu\eta_1)^\sharp\cdot \fa=(\lambda\mu\eta_2)^\sharp\cdot \fa=0.
\end{equation}
This shows that $\chi(\fa)$ divides sufficiently high powers of both $(\lambda\mu\eta_1)^\sharp$ and $(\lambda\mu\eta_2)^\sharp$.
But since $\eta_1^{\sharp}$ and $\eta_2^\sharp$ are coprime, they must be both coprime to $\chi(\fa)$.

Set $\fc=(\mu\eta_1)^\sharp\cdot \fa$, $\fc_n=\fk_n(\fc)$ for each $n$, and form the $\Gamma$-systems $\fC$, $\fE$ by the construction in \S\ref{su:de}.
Let $\fd$, $\fe$, and $\ff$ be as in \eqref{e:c} and \eqref{e:e}.
Since $\fk_n(\fb)=\fb_n$, we have $\fk_n(\fd)=\fd_n$, and hence, by Lemma \ref{l:eq}, $\fC$ is also strongly controlled.
By \eqref{e:918}, $\lambda^\sharp\cdot \fc=0$, whence $\lambda\cdot\fd=0$ thanks to Lemma \ref{l:an}. Then case (1) of Theorem \ref{t:al} says $\fc^\sharp\sim \fd$.
To complete the proof it is sufficient to show that
$$\xymatrix{[\fb]_{ns}\times [\fa]_{ns}\ar[r]^-{\varphi\times\psi} & \fd\times \fc}$$
(where $\varphi$ and $\psi$ are respectively the restrictions to $[\fb]_{ns}$ and $[\fa]_{ns}$ of the projection $\fb\rightarrow \fd=\fb/\ff$ and of the multiplication by $(\mu\eta_1)^\sharp$ on $\fa$) is a pseudo-isomorphism.

The inclusion $\mu\eta_1\cdot \fb\subset \mu\cdot [\fb]\subset[\fb]_{ns}$ implies $\mu\eta_1\cdot \Coker ( \varphi)=0$. Furthermore, since $\lambda\cdot \Coker (\varphi)$ is a quotient of $\lambda\cdot\fd=0$, it must be trivial. Thus, by Lemma \ref{l:psn}, $\Coker (\varphi)$, being annihilated by coprime $\lambda$ and $\mu\eta_1$, is pseudo-null.
Next, we observe that $(\mu\eta_1)^\sharp\cdot \fe=(\mu\eta_1)^\sharp\cdot \fa/\fc=0$ yields $(\mu\eta_1)^\sharp\cdot \fe_n=0$. The duality implies that each $\ff_n$ is annihilated by $\mu\eta_1$, and by taking the projective limit we see that $\ff$ is also annihilated by $\mu\eta_1$.
It follows that $\Ker (\varphi)=[\fb]_{ns}\cap \ff=0$ since no nontrivial element of $[\fb]_{ns}$ is annihilated by $\mu\eta_1$ (because $\eta_1$ is coprime to $\chi(\fb)$ while $\mu$ is a product of simple elements). Similarly, $\Ker(\psi)=0$ since no nontrivial element of $[\fa]_{ns}$ is annihilated by $(\mu\eta_1)^\sharp$.
To show that $\Coker (\psi)$ is pseudo-null, we choose an $\eta_3\in\La$, coprime to $\lambda\eta_1$, such that $\eta_3^\sharp\cdot \fa\subset [\fa]$. Then \eqref{e:918} together with the fact that $\lambda$ is non-simple imply that $(\mu\eta_1\eta_3)^\sharp \cdot \fa\subset (\mu\eta_1)^\sharp \cdot [\fa]\subset[\fa]_{ns}$. This implies $(\mu\eta_1\eta_3)^\sharp \cdot \Coker (\psi)=0$.
Since $\lambda^\sharp\cdot\Coker (\psi)$, being a quotient of $\lambda^\sharp\cdot \fc=0$, is trivial and $\lambda^\sharp$, $(\mu\eta_1\eta_3)^\sharp$ are coprime, the proof is completed.
\end{proof}

\subsubsection{Twists of $\Gamma$-systems} \label{ss:twistgam}
Recall that associated to a continuous group homomorphism $\phi\colon\Gamma\rightarrow\ZZ_p^\times$, there is the ring isomorphism $\phi^*\colon\La\rightarrow\La$ defined in \S\ref{ss:iwH1}. Given such a $\phi$ and a $\Gamma$-system $\fA$, we can form
$$\fA(\phi):=\{\fa_n(\phi^{-1}),\fb_n(\phi),\langle\;,\;\rangle^\phi_n,\fr(\phi)_m^n,\fk(\phi)_m^n\;\mid\; n,m\in\NN\cup\{0\},\; n\geq m\}\,,$$
where $\fa_n(\phi^{-1})$ and $\fb_n(\phi)$ are twists as defined in \eqref{e:twistbyphi},
$$\langle x\otimes a_n, y\otimes b_n\rangle^\phi_n:=\langle \phi^*(x) a_n, ({\phi^{-1}})^*(y) b_n\rangle_n$$
and $\fr(\phi)_m^n$, $\fk(\phi)_m^n$ are respectively the maps induced by $1\otimes\fr_m^n$ and, $1\otimes\fk_m^n$. In general $\fA(\phi)$ won't be a $\Gamma$-system, because the action of $\Gamma$ on $\fa_n(\phi^{-1})$, $\fb_n(\phi)$ does not factor through $\Gamma_n$. However if we take $\fA$ twistable of order $k$ and $\phi$ such that
\begin{equation}\label{e:phicondition}
\phi(\Gamma)\subseteq 1+p^k\ZZ_p,
\end{equation}
then both $\fa_n(\phi^{-1})$ and $\fb_n(\phi)$ are still $\Gamma_n$-modules, because $\phi(\Gamma^{(n)})\subset 1+p^{n+k}\ZZ_p$ by \eqref{e:phicondition} and $p^{n+k}\fa_n=0$.

\begin{lemma}\label{l:nonsimpletwist}
For any $k\in\NN$ and $\xi\in\La-\{0\}$, there exists a continuous group homomorphism $\phi\colon\Gamma\rightarrow\ZZ_p^\times$ such that \eqref{e:phicondition} holds and both $\phi^*(\xi)$ and $({\phi^{-1}})^*(\xi)$ are not divisible by simple elements.
\end{lemma}

\begin{proof} First of all, note that $({\phi^{-1}})^*(\xi)$ is not divisible by any simple element if and only if the same holds for $({\phi^{-1}})^*(\xi)^\sharp=\phi^*(\xi^\sharp)$. So we just need to find $\phi$ such that $\phi^*(\xi\xi^\sharp)$ has no simple factor.
An abstract proof of the existence of such $\phi$ can be obtained by the Baire category theorem, observing that if $\lambda\in\La-\{0\}$ then $\Hom(\Gamma,\ZZ_p^\times)$ cannot be contained in $\cup_\omega\Ker\big(\phi\mapsto\omega(\phi^*(\lambda))\big)$, since all these kernels have empty interior. A more concrete approach is the following.

Call an element $\lambda\in\La$ a simploid if it has the form $\lambda=u\cdot f_{\gamma,\beta}$ where $u\in\Lambda^\times$ and
$$f_{\gamma,\beta}:=\prod_{\sigma\in \Gal(\QQ_p(\beta)/\QQ_p)} (\gamma-\sigma(\beta))$$
with $\gamma\in\Gamma-\Gamma^p$ and $\beta$ a unit in some finite Galois extension of $\QQ_p$. Simploids are easily seen to be irreducible, so by unique factorization any principal ideal $(\lambda)\subset\Lambda$ can be written as $(\lambda)=(\lambda)_s(\lambda)_n$ with no simploid dividing $(\lambda)_n$. Moreover, given any $\phi\colon\Gamma\rightarrow\ZZ_p^\times$, the equality
$$\phi^*(f_{\gamma,\beta})=\phi(\gamma)^{-[\QQ_p(\beta):\QQ_p]}\cdot f_{\gamma,\phi(\gamma)\beta}$$
shows that the set of simploids is stable under the action of $\phi$ and $(\phi^*(\lambda))_s=\phi^*((\lambda)_s)$.
Thus, if $f_{\gamma_1,\beta_1},...,f_{\gamma_l,\beta_l}$ is a maximal set of coprime simploid factors of $\xi\xi^\sharp$ and if $\phi$ is chosen such that no $\phi(\gamma_i)\beta_i$, $i=1,...,l$, is a root of unit, then $\phi^*(\xi\xi^\sharp)$ is not divisible by any simple element.
\end{proof}

\begin{proof}[{\bf Proof of Theorem \ref{t:al}(2)}]
Let $\xi$ be a generator of $\chi(\fa)\chi(\fb)$ and let $\phi$ be as in Lemma \ref{l:nonsimpletwist}.
Then $\fA(\phi)$ also form a pseudo-controlled $\Gamma$-system with $\fa(\phi^{-1})=\varprojlim_n\fa_n(\phi^{-1})$ and $\fb(\phi)=\varprojlim_n\fb_n(\phi)$. By Lemma \ref{l:phi[]chi}, both $\chi(\fa(\phi^{-1}))$ and $\chi(\fb(\phi))$ are not divisible by simple elements, and hence $[\fa(\phi^{-1})]^\sharp=[\fa(\phi^{-1})]_{ns}^\sharp$ and $[\fb(\phi)]=[\fb(\phi)]_{ns}$. Therefore,
$$[\fa]^{\sharp}=[\fa(\phi^{-1})](\phi)^{\sharp}=[\fa(\phi^{-1})]^\sharp(\phi^{-1})=[\fb(\phi)](\phi^{-1})=[\fb],$$
where the first and the last equality are consequence of Lemma \ref{l:phi[]chi} and the third follows from Theorem \ref{t:al}(1) applied to $\fA(\phi)$.
\end{proof}

\end{subsection}

\subsubsection{Complete $\Gamma$-systems}
Now we assume that our original $\fA$ is just a part of a complete $\Gamma$-system which we still denote by $\fA$. The original $\fA$ is pseudo-controlled if and only if so is its complete system. Also, if the original $\fA$ is strongly controlled, then by replacing $\fa_F\times \fb_F$ by $\mathfrak{k}_F(\fa\times \fb)$ we can make the complete system strongly controlled without altering $\fa$ and $\fb$. So we shall assume that $\fA$ is strongly controlled.\\

First we assume that $\fa$ is annihilated by a simple element $\xi=f_{\gamma_1,\zeta}$ and extend $\gamma_1$ to a basis $\gamma_1,...,\gamma_d$ of $\Gamma$ over $\ZZ_p$.
Let $\Psi$ and $\Gamma'$ be the subgroups of $\Gamma$ with topological generators respectively $\gamma_1$ and $\{\gamma_2,...,\gamma_d\}$. Note that for $H\subset\Gamma$ a closed subgroup we shall write $H^{(n)}$ for $H^{p^n}$. Let $K_{n',n}$ denote the fixed field of the subgroup $\Psi^{(n)}\oplus(\Gamma')^{(n')}$ and write $\fa_{\infty,n}:=\varprojlim_{n'} \fa_{n',n}$, $\fb_{\infty,n}:=\varprojlim_{n'} \fb_{n',n}$ with the obvious meaning of indexes. They are $\Lambda$-modules. Let $K_{\infty,n}$ denote the subfield of $L$ fixed by  $\Psi^{(n)}$. Then the restriction of Galois action gives rise to a natural isomorphism $\Gamma'\simeq\Gal(K_{\infty,n}/K_{0,n})$. Write $\La':=\La(\Gamma')$. We shall view $\La'$ as a subring of $\La$.

Since $\fA$ is strongly controlled, $\fa_{\infty,n}=\fk_{\infty,n}(\fa)$ and $\fb_{\infty,n}=\fk_{\infty,n}(\fb)$ are finitely generated over $\La$, and hence finitely generated over $\La'$, because they are fixed by $\Psi^{(n)}$.

\begin{proposition}\label{p:twistable} Suppose $\fA$ is a strongly controlled complete $\Gamma$-system such that \begin{enumerate}
\item $\fa$ and $\fb$ are annihilated by the simple element $\xi=f_{\gamma_1,\zeta}$ defined above, with $\zeta$ of order $p^l$;
\item $\fa_{\infty,m}$ and  $\fb_{\infty,m}$ are torsion over $\La'$ for some $m\geq l$. \end{enumerate}
Then there exists some non-trivial $\eta\in\Lambda'$ such that $\eta\cdot \fA$ is twistable.
\end{proposition}

Here $\eta\cdot\fA$ is the complete $\Gamma$-system as defined in Example \ref{eg:morf}. It is also strongly controlled if so is $\fA$.

\begin{proof} Since $\zeta$ is of order $p^l$, the action of $\gamma_1^{p^l}$ is trivial on both $\fa_{\infty,n}$ and $\fb_{\infty,n}$  for all $n$. Assume that $m\geq l$ and suppose both $\fa_{\infty,m}$ and $\fb_{\infty,m}$ are annihilated by some non-zero $\eta\in\La'$. Then $\eta\cdot \fa_{n',m}=0$ and $\eta\cdot \fb_{n',m}=0$ for all $n'$. Hence for $n\geq m$,
$$p^{n-m}\eta\fa_{n',n}=\fr^{n',n}_{n',m}(\fk^{n',n}_{n',m}(\eta\fa_{n',n}))=0$$
since $\gamma_1^{p^n}$ acts trivially on $\fa_{n',n}$. In particular, $p^{n-m}\eta\cdot\fa_n=0$ and by similar argument
$p^{n-m}\eta\cdot\fb_n=0$. Then choose $k$ such that $p^k\fa_i=p^k\fb_i=0$ for each $1\leq i< m$.
\end{proof}

\begin{corollary}\label{c:twistable}
Suppose $\fA$ satisfies the condition of {\em Proposition \ref{p:twistable}}. Then
$$\fa^\sharp\sim \fb.$$
\end{corollary}

\begin{proof} The morphism $\fA\rightarrow\eta\cdot\fA$ of Example \ref{eg:morf} in this case is a pseudo-isomorphism, because $\fa[\eta]$ and $\fb/\eta^\sharp\fb$ are both killed by $f_{\gamma_1,\zeta}$ and either $\eta$ or $\eta^\sharp$. Now apply Theorem \ref{t:al}(2). \end{proof}

\begin{proof}[{\bf Proof of Theorem \ref{t:al}(3)}] We may assume that $\fA$ is strongly controlled. Suppose $\fa$ is annihilated by $\xi\in\Lambda$, and hence $\fb$ is annihilated by $\xi^\sharp$. Write $\xi=\xi_1^{s_1}\cdot \cdots \cdot \xi_k^{s_k}$, where each $\xi_i$ is irreducible and $s_i$ is a positive integer. The proof is by induction on $k$.

First assume $k=1$. If $\xi$ is non-simple, then the theorem has been proved. Thus, we may assume that $\xi_1$ is simple and we proceed by induction on $s_1$. The case $s_1=1$ is Corollary \ref{c:twistable}.
If $s_1>1$ let $\fc_F:=\xi_1\cdot \fa_F$ and form the derived systems $\fC$ and $\fE$ as in \S\ref{su:de}. Note that both enjoy property ({\bf T}), as immediate from the sequences  \eqref{e:c} and \eqref{e:e}. Besides $\fC$ is strongly controlled and $\fc$ is annihilated by $\xi_1^{s_1-1}$, whence (as $\xi_1$ is simple) $[\fc]=[\fc]^\sharp=[\fd]$
by the induction hypothesis. We still have $\ff^0=0$, but we don't know if $\fe^0=0$. However, induction tells us that $[\fe/\fe^0]=[\ff]$, or equivalently, there is an injection $[\ff]\hookrightarrow [\fe]$. This actually implies an inclusion $[\fb]\hookrightarrow[\fa]$: to see it, write
$$[\fa]=(\La/\xi_1\La)^{a_1}\oplus (\La/\xi_1^2\La)^{a_2}\oplus \cdots \oplus (\La/\xi_1^{s_1}\La)^{a_{s_1}},$$
and
$$[\fb]=(\La/\xi_1\La)^{b_1}\oplus (\La/\xi_1^2\La)^{b_2}\oplus \cdots \oplus (\Lambda/\xi_1^{s_1}\La)^{b_{s_1}}.$$
Then
$$[\fc]= (\Lambda/\xi_1\La)^{a_2}\oplus \cdots \oplus (\Lambda/\xi_1^{s_1-1}\La)^{a_{s_1}},$$
and
$$[\fd]=(\Lambda/\xi_1\La)^{b_2}\oplus \cdots \oplus (\Lambda/\xi_1^{s_1-1}\La)^{b_{s_1}},$$
while
$$[\fe]= (\Lambda/\xi_1\La)^{a_1+a_2+\cdots +a_{s_1}},\; \;[\ff]= (\Lambda/\xi_1\La)^{b_1+b_2+\cdots +b_{s_1}}. $$
Thus, we have $a_1\geq b_1$ and $a_i=b_i$ for $1<i\leq s_1$. Then by symmetry, we also have $[\fa]\hookrightarrow[\fb]$, whence $[\fa]=[\fb]$ as desired. This proves the $k=1$ case.

For $k>1$, form again $\fC$ and $\fE$, this time setting $\fc_F:=\xi_1^{s_1}\fa_F$. Then induction yields $[\fc]^\sharp=[\fd]$ and $[\fe]^\sharp=[\ff]$. To conclude, use the decompositions $[\fa]=[\fc]\oplus[\fe]$, $[\fb]=[\fd]\oplus[\ff]$ which hold because in the sequences \eqref{e:c}, \eqref{e:e} the extremes have coprime annihilators.
\end{proof}

\end{section}

%SECTION 4++++++++++++++++++++++++++++++++++++++++++++++++++++++++++++++++++++

\begin{section}{Cassels-Tate systems of abelian varieties}\label{s:ctsys}
From now on, $K$ will be a global field, $L/K$
a $\ZZ_p^d$-extension with a finite ramification locus denoted by $S$, and $\Gamma=\Gal(L/K)$.
Let $A/K$ an abelian variety which has \textit{potentially ordinary} reduction at every place in $S$.

In this section we construct from Selmer groups of abelian varieties
over global fields Cassels-Tate systems to which we apply the theory of
Pontryagin duality for Iwasawa modules given earlier.

\begin{subsection}{The Selmer groups}\label{se:se}
Let $i\colon A_{p^n}\hookrightarrow A$ be the group scheme of $p^n$-torsion of $A$.
The $p^n$-Selmer group $\Sel_{p^n}(A/K)$ is defined to be the kernel of the composition
\begin{equation}\label{e:nselmer} \begin{CD} \coh^1_{\mathrm{fl}}(K,A_{p^n}) @>{i^*}>> \coh^1_{\mathrm{fl}}(K,A)@>{loc_K}>> \bigoplus_v \coh^1_{\mathrm{fl}}(K_v,A)\,, \end{CD}\end{equation}
where $\coh_{\mathrm{fl}}^\bullet$ denotes the flat cohomology and $loc_K$ is the localization map to the direct sum of local cohomology groups over all places of $K$.
The same definition works over any finite extension $F/K$. Taking the direct limit as $n\rightarrow\infty$, we get
\begin{equation}\label{e:selmer} \Sel_{p^\infty}(A/F):=\Ker\big(\coh^1_{\mathrm{fl}}(F,A_{p^\infty})\lr \bigoplus_{\text{all}\ v}\coh^1_{\mathrm{fl}}(F_v,A)\big) \end{equation}
where $A_{p^{\infty}}$ is the $p$-divisible group associated with $A$.
The Selmer group sits in an exact sequence
\begin{equation} \label{e:selmersha} 0 \lra \QQ_p /\ZZ_p\otimes_{\ZZ}A(F) \lra \Sel_{p^\infty}(A/F) \lra \Sha_{p^\infty}(A/F) \lra 0 , \end{equation}
where $\Sha_{p^\infty}(A/F)$ denote the $p$-primary part of the Tate-Shafarevich group
$$\Sha(A/F) := \Ker\big( \coh^1_{\mathrm{fl}}(F,A)\lr\bigoplus_v \coh^1_{\mathrm{fl}} (F_v,A)\big) .$$
Also, let $\Sel_{p^{\infty}}(A/F)_{div}$ denote the $p$-divisible part of $\Sel_{p^{\infty}}(A/F)$
and write $\M(A/F)$ for $\QQ_p/\ZZ_p\otimes_{\ZZ}\,A(F)$.
We have $\M(A/F)\subset \Sel_{p^{\infty}}(A/F)_{div}\subset \Sel_{p^{\infty}}(A/F)$.
\begin{definition}\label{d:xyz} Define
$$\Sel_{p^\infty}(A/L):=\varinjlim_{F}\Sel_{p^{\infty}}(A/F),$$
and
$$\Sel_{div}(A/L):=\varinjlim_{F}\Sel_{p^{\infty}}(A/F)_{div}.$$
%and
%$$\M(A/L):=\varinjlim_F \M(A/L).$$
Let $X_p(A/L)$ and $Y_p(A/L)$ %, and $Z_p(A/L)$
denote the Pontryagin dual of $\Sel_{p^\infty}(A/L)$ and $\Sel_{div}(A/L)$.
%and $\M(A/L)$.
\end{definition}

The Galois group $\Gamma$ acts on the above modules turning them into $\La$-modules.
We point out that $X_p(A/L)$
is finitely generated over $\La$, and hence so is $Y_p(A/L)$. % and $Z_p(A/L)$.
In the case where $A$ has ordinary reduction at all places in $S$, this is \cite{tan10b}, Proposition 1.2.1 and Corollary 2.4.2.
To pass from potentially ordinary reduction to ordinary reduction, one can argue as in \cite[Lemma 2.1]{ot09}.

\subsubsection{$Y_p(A/L)$}
The following theorem was originally proved in \cite{tan10b} under the assumption of ordinary reduction. Here
we prove a much more general version: Theorem \ref{t:general}.
\begin{mytheorem}[Tan] \label{t:flat}
Suppose $X_p(A/L)$ is a torsion $\La$-module. Then there exist relatively prime simple elements $f_1,...,f_m$  {\em (}$m\geq 1${\em )} such that
$$f_1\cdots f_m\cdot \Sel_{div}(A/L)=0.$$
\end{mytheorem}

\begin{mytheorem}\label{t:general}
Let $M$ be a cofinitely generated torsion $\La$-module. Then there exist relatively prime simple elements $f_1,...,f_m$  {\em (}$m\geq 1${\em )} such that
$$f_1\cdots f_m\cdot (M^{\Gal(L/F)})_{div}=0$$
for any finite intermediate extension $K\subset F\subset L$.
\end{mytheorem}

\begin{proof}
By hypothesis, there is some $\xi\in\La-\{0\}$ annihilating $M$. Let $f\in\La$ be such that $\omega(f)=0$ for all $\omega\in\Delta_\xi\,$. Fix a finite intermediate extensions $F/K$ and, to lighten notation, put $G:=\Gal(F/K)$ and $\cO$ the ring of integers of $\QQ(\boldsymbol\mu_{p^d})$, where $p^d$ is the exponent of $G$. Also, let $N:=\cO\otimes_{\ZZ_p}M^{\Gal(L/F)}$. Then defining, as in \eqref{e:idempotent},  $e_\omega':=\sum_{g\in G}\omega(g^{-1})g$ for each $\omega\in G^\vee$, one finds $|G|\cdot N=\sum e_\omega'N$ and $\La$ acts on $e_\omega'N$ by $\lambda\cdot n=\omega(\lambda)n$. In particular, one has $f\cdot e_\omega'N=0$ for all $\omega\in\Delta_\xi$. On the other hand, if $\omega\notin\Delta_\xi$ then $e_\omega'N$ is finite because it is a cofinitely generated module over the finite ring $\omega(\La)/(\omega(\xi))$. It follows that $f\cdot N$ is finite. Since a finite divisible group must be trivial, this proves that $f\cdot N_{div}=0$, and hence $f\cdot (M^{\Gal(L/F)})_{div}=0$.

It remains to prove that one can find a product of distinct simple elements which is killed by $\Delta_\xi$. This is a consequence of Monsky's theorem: one has $\Delta_\xi=\cup T_j$ where the $T_j$'s are $\ZZ_p$-flats, and by definition for each $T_j$ there is at least one simple element vanishing on it.
\end{proof}

\begin{corollary}\label{c:y} If $X_p(A/L)$ is torsion, then there is a pseudo-isomorphism
$$Y_p(A/L)\sim \bigoplus_{i=1}^m (\La/f_i\La)^{r_i},$$
for some non-negative integers $r_i$, and hence
$$[Y_p(A/L)]=[Y_p(A/L)]_{si}=[Y_p(A/L)]_{si}^\sharp=[Y_p(A/L)]^\sharp.$$
\end{corollary}
\begin{proof}
The second equality is by \eqref{e:fld}.
\end{proof}
\end{subsection}

\begin{subsection}{The Cassels-Tate system}\label{su:cts}
\subsubsection{The Cassels-Tate pairing} \label{CTpairing}
%The main reference of this section is \cite[II.2 and II.5]{mil86a}.
For an abelian variety $A$ defined over the global field $K$  let $A^t$ its dual abelian variety.
Let
$$\langle\;,\;\rangle_{A/K}\colon \Sha(A/K)\times \Sha(A^t/K)\longrightarrow \QQ/\ZZ$$
denote the Cassels-Tate pairing (\cite[II.5.7 (a)]{mil86}). %(We shall recall its construction below.)

The next proposition follows from \cite[I, Theorem 7.3]{mil86}.
\begin{proposition}\label{p:1}
Let $A$, $B$ be abelian varieties defined over the global field $K$.
Suppose $\phi\colon A\rightarrow B$ is an isogeny and $\phi^t\colon B^t\rightarrow A^t$ is its dual.
Then we have the commutative diagram:
\[ \begin{CD} \langle\;,\;\rangle_{A/K}\colon & \Sha(A/K) &\,\times & \,\Sha(A^t/K)  @>>> \QQ/\ZZ \\
 &  @VV{\phi_*}V @AA{\phi^t_*}A @| \\
\langle\;,\;\rangle_{B/K}\colon & \Sha(B/K)&\,\times  & \,\Sha(B^t/K)  @>>> \QQ/\ZZ.
\end{CD} \]
\end{proposition}

%=============================================================================

\subsubsection{The Cassels-Tate system}\label{sb:cts}

Let $A$ be an abelian variety defined over the global field $K$.
Put
\begin{equation} \label{e:ansha} \fa_n=\Sha_{p^\infty}(A^t/K_n)/\Sha_{p^\infty}(A^t/K_n)_{div}=\Sel_{p^{\infty}}(A^t/K_n)/\Sel_{p^{\infty}}(A^t/K_n)_{div} \,, \end{equation}

\begin{equation} \label{e:bnabvar} \fb_n=\Sel_{p^{\infty}}(A/K_n)/\Sel_{p^{\infty}}(A/K_n)_{div}
=\Sel_{p^{\infty}}(A/K_n)/\Sel_{p^{\infty}}(A/K_n)_{div}. \end{equation}
Let
\begin{equation} \label{e:perfectCT}
\langle\;,\;\rangle_n\colon \fa_n\times\fb_n \lr \QQ_p/\ZZ_p \end{equation}
be the perfect pairing
induced from the Cassels-Tate pairing on $\Sha(A^t/K_n)\times \Sha(A/K_n)$.
Let $\fr_m^n$ and $\fk_m^n$ be the morphisms induced respectively from the restriction
$$\coh^1(K_m,A^t\times A)\longrightarrow \coh^1(K_n,A^t\times A)$$
and the co-restriction
$$\coh^1(K_n,A^t\times A)\longrightarrow \coh^1(K_m,A^t\times A).$$
 Let $$\fA =\{\fa_n,\fb_n,\langle\,,\rangle_n,\fr_m^n,\fk_m^n\}.$$
  We call $\fA$  the {\em{Cassels-Tate}} system of $A$.
  As before we write
  $$\fa:=\varprojlim_{n}\fa_n\;\;
\text{and}\;\;\fb:=\varprojlim_{n}\fb_n\,.$$
It is clear that $\fA$ satisfies axioms ($\Gamma$-1)-($\Gamma$-4).
For the rest of this paper we shall use the above notations.

Theorem \ref{t:ct} says that $\fA$ is a $\Gamma$-system,
but for the time being, we do not need it.
As in Section \ref{s:con}, write $\fa$ and $\fb$ for the projective limits of $\{\fa_n\}_n$ and $\{\fb_n\}_n$. We have the exact sequence
\begin{equation}\label{e:axy}  0\lr \fa\lr X_p(A/L)\lr Y_p(A/L)\lr 0. \end{equation}
\begin{lemma}\label{l:axsharp} If $X_p(A/L)$ is torsion, then $[X_p(A/L)]_{ns}=[\fa]_{ns}$.
\end{lemma}
\begin{proof}
Corollary \ref{c:y} and Equation \eqref{e:axy}.
\end{proof}

\subsubsection{The module $\fa^{00}$}\label{ss:ddiv} To study $\fa_n^0$, we first consider the small piece $\fa^{00}_n$ which is the image of
$$\fs_n^{00}:=\Ker\!\big(\Sel_{p^{\infty}}(A^t/K_n)\,\lr \Sel_{p^{\infty}}(A^t/L)\big)$$
under the projection $\Sel_{p^{\infty}}(A^t/K_n)\twoheadrightarrow \fa_n$. Obviously, $\fs_n^{00}$ is a $\Gamma_n$-submodule of
$$\fs_n^0:=\Ker\big(\coh_{\mathrm{fl}}^1(K_n,A^t_{p^{\infty}})\,\lr \coh_{\mathrm{fl}}^1(L,A^t_{p^{\infty}})\big)=\coh^1(\Gamma^{(n)}, A^t_{p^{\infty}}(L))\,.$$

\begin{lemma} \label{l:h1d}
All the groups $\coh^1\!\big(\Gamma^{(n)},A^t_{p^{\infty}}(L)\big)$ and $\coh^2\!\big(\Gamma^{(n)},A^t_{p^{\infty}}(L)\big)$ are finite.
\end{lemma}

\begin{proof} It follows from \cite[Proposition 3.3]{gr03}. Here we partially prove the $d=1$ case, because some ingredient of it will be applied later. Write
\begin{equation}\label{e:ALtors} D:=A^t_{p^{\infty}}(L)=A^t(L)[p^{\infty}] \end{equation}
and let $D_{div}$ be its $p$-divisible part. We have an exact sequence
\begin{equation}\label{e:hddd}
(D/D_{div})^{\Gamma^{(n)}} \lr \coh^1(\Gamma^{(n)},D_{div})\,\lr \coh^1(\Gamma^{(n)},D) \,\lr \coh^1(\Gamma^{(n)}, D/D_{div})\,.
\end{equation}
If $d=1$ and $\Gamma=\gamma^{\ZZ_p}$, then we observe that
$$(\gamma^{p^n}-1)D_{div}=D_{div},$$
since $\Ker(D\stackrel{\gamma^{p^n}-1}{\lr} D)=D^{\Gamma^{(n)}}$ is finite. Therefore, $\coh^1(\Gamma^{(n)},D_{div})=0$
and hence, since $D/D_{div}$ is finite, from the exact sequence \eqref{e:hddd} we deduce (see e.g.\! \cite[Lemma 4.1]{bl09})
\begin{equation}\label{e:ddd}
|\coh^1(\Gamma^{(n)},D)|\leq |D/D_{div}|.
\end{equation}
\end{proof}

\begin{lemma}\label{l:a00}
The projective limit
$$\fa^{00}:=\varprojlim\fa_n^{00}$$
is pseudo-null.
\end{lemma}

\begin{proof} If $d=1$, $\fa^{00}$ is pseudo-null since it is a finite set: inequality \eqref{e:ddd} together with the surjection $\fs^{00}_n\twoheadrightarrow\fa^{00}_n$ gives a bound on its cardinality.

Let $D$ be as in \eqref{e:ALtors} and let $r$ the $\ZZ_p$-rank of its Pontryagin dual $D^\vee$.
Then the action of $\Gamma$ gives rise to a representation
\begin{equation} \label{e:ro} \rho\colon\Gamma\lr \Aut(D_{div})\simeq \GL(r,\ZZ_p)\,. \end{equation}
For each $\gamma$, let $f_{\gamma}(x)$ be the characteristic polynomial of $\rho(\gamma)$. Then $f_{\gamma}(\gamma)$ is an element in $\La$ which annihilates $D_{div}$. Let $m_0$ be large enough such that $A^t(K_{m_0})[p^\infty]$ generates $D/D_{div}$. Then for every $\gamma$, $(\gamma^{p^{m_0}}-1)f_{\gamma}(\gamma)$ annihilates both $\fs_n^{00}$ and $\fa_n^{00}$ for $n\geq m_0$, since we have
$$(\gamma^{p^{m_0}}-1)f_{\gamma}(\gamma)\cdot D=0.$$
If $d\geq 2$ and $\Gamma=\bigoplus_{i=1}^d\gamma_i^{\ZZ_p}$, then each $\delta_i:=(\gamma_i^{p^{m_0}}-1)f_{\gamma_i}(\gamma_i)$ lives in a different $\ZZ_p[T_i]$ under the identification $\Lambda=\ZZ_p[[T_1, ...,T_d]]$,  $\gamma_i-1=T_i$. Therefore $\delta_1,...,\delta_d$ are relatively prime and $\fa^{00}$ is pseudo-null by Lemma \ref{l:psn}.
\end{proof}

Put ${\bar \fa}_n^0:=\fa_n^0/\fa_n^{00}$. By construction we have an exact sequence
\begin{equation}\label{e:aabar}
0\longrightarrow \fa^{00}\longrightarrow \fa^0\longrightarrow {\bar \fa}^0:=\varprojlim_n{\bar \fa}_n^0\longrightarrow 0,
\end{equation}
and hence $\fa^0\sim {\bar \fa}^0$.
Applying the snake lemma to the diagram $$\begin{CD}
0 @>>> \Sel_{p^{\infty}}(A^t/K_n)_{div} @>>> \Sel_{p^{\infty}}(A^t/K_n) @>>> \fa_n @>>> 0\\
&& @VVV @VVV @VVV\\
0 @>>> \Sel_{div}(A^t/L) @>>> \Sel_{p^{\infty}}(A^t/L) @>>> \varinjlim\fa_n @>>> 0\end{CD}$$
we find an injection
\begin{equation}\label{e:subset}
{\bar \fa}_n^0\hookrightarrow \Sel_{div}(A^t/L)/\Sel_{p^{\infty}}(A^t/K_n)_{div}\,.
\end{equation}

\subsection{The proof of Theorem \ref{t:xaat}}\label{sb:fleq}
First we consider the case where $X_p(A/L)$ is a torsion $\La$-module. %Then so are $\fa$ and, by the following lemma, $\fb$.

\begin{lemma} \label{l:isogAtA}
The $\La$-module $X_p(A^t/L)$ is torsion if and only if so is $X_p(A/L)$.
\end{lemma}

\begin{proof}
Any isogeny $\varphi\colon A\rightarrow A^t$ defined over $K$ gives rise to a homomorphism of $\La$-modules
$\varphi_* \colon X_p(A^t/L)\rightarrow X_p(A/L)$ with kernel and co-kernel annihilated by $\deg(\varphi)$.
\end{proof}

\begin{corollary}\label{c:isogAtA} If $X_p(A^t/L)$ is torsion over $\La$, then $[X_p(A/L)]_{si}=[X_p(A^t/L)]_{si}$
and $[Y_p(A/L)]=[Y_p(A^t/L)]$.
\end{corollary}
\begin{proof}We claim that the composition
$$\alpha:\xymatrix{[X_p(A/L)]_{si}\ar@{^{(}->}[r] & X_p(A/L) \ar[r]^-{\alpha_1} & X_p(A^t/L) \ar[r]^-{\alpha_2} & [X_p(A^t/L)] \ar[r]^-{\alpha_3} & [X_p(A^t/L)]_{si},}$$
where $\alpha_1=\varphi_*$, $\alpha_2$ is a pseudo-isomorphism and $\alpha_3$ is the projection, is a pseudo-injection. By symmetry, there is also
a pseudo injection from $[X_p(A^t/L)]_{si}$ to $[X_p(A/L)]_{si}$. Then the first equality follows from
Lemma \ref{l:pseudoisom}. Suppose $[X_p(A/L)]_{si}$ is annihilated by $f\in\La$, which is a product of simple elements.
To prove the claim, we note that the kernel of each $\alpha_i$ is annihilated by certain element $g_i\in \La$ relatively
prime to $f$. Thus the kernel of $\alpha$ is annihilated by both $f$ and $g_1g_2g_3$. Then apply Lemma \ref{l:psn}.

The second equality can be proved similarly, since $\varphi_*$ sends $Y_p(A/L)$ to $Y_p(A^t/L)$ and
$[Y_p(A/L)]_{si}=[Y_p(A/L)]$, $[Y_p(A^t/L)]_{si}=[Y_p(A^t/L)]$ by Corollary \ref{c:y}.

\end{proof}

\begin{lemma}\label{l:flat} If $X_p(A/L)$ is torsion over $\La$, then
$$[\fa]_{ns}^\sharp=[\fb]_{ns}.$$
\end{lemma}
\begin{proof}
By \eqref{e:axy}, $\fa$ is torsion, and similarly so is $\fb$.
Let $f_1,...,f_m$ be those simple elements in Theorem \ref{t:flat} (applied to $A^t$).
By \eqref{e:subset}, we have $f_1\cdots f_m\cdot \bar{\fa}^0=0$, and hence $[\bar{\fa}^0]=[\bar{\fa}^0]_{si}$.
Then Lemma \ref{l:a00} and \eqref{e:aabar} together imply that $[\fa^0]=[\fa^0]_{si}$.
Hence, by \eqref{e:0'}, we have $[\fa]_{ns}=[\fa']_{ns}$. Similarly, $[\fb]_{ns}=[\fb']_{ns}$.
Then apply Corollary \ref{c:nons}.

\end{proof}

\begin{proposition}\label{p:xtor}
If $X_p(A/L)$ is torsion over $\La$, then
$$[X_p(A/L)]^\sharp=[ X_p(A^t/L)]=[ X_p(A/L)]= [X_p(A^t/L)]^\sharp.$$
\end{proposition}
\begin{proof}
Lemma \ref{l:axsharp} and Lemma \ref{l:flat} imply $[X_p(A/L)]_{ns}^\sharp =[X_p(A/L)]_{ns}$,
while Corollary \ref{c:isogAtA} says $[X_p(A/L)]_{si}^\sharp =[X_p(A/L)]_{si}$.
Thus, the first equality is proved, and the third one is proved by a similar argument.
To proceed, for a finitely generated torsion $\La$-module $M$,
we define a decomposition in
\textit{$p$ part} and \textit{non-$p$ part}
$$[M]=[M]_{p}\oplus [M]_{np},$$
in the following way: if $[M]$ is a direct sum of components $\La/\xi_i^{r_i}\La$, we define $[M]_{p}$ as the sum over those with
$\xi_i=(p)$ and
$[M]_{np}$ as its complement. We have $[M]_p^\sharp=[M]_p$.
This and the first equality imply
$$[X_p(A/L)]_p=[X_p(A/L)]_p^\sharp=[X_p(A^t/L)]_p.$$

Let $\varphi$ be the isogeny in the proof of Lemma \ref{l:isogAtA} and write $\deg(\varphi)=p^n\cdot u$ with $u$ relatively prime to $p$.
Then the kernel and the cokernel of $\varphi_*$ are annihilated by $p^n$. This then leads to
$$[X_p(A/L)]_{np}=[X_p(A^t/L)]_p.$$

\end{proof}

\noindent

\begin{proof}[{\bf Proof of Theorem \ref{t:xaat}}]
 The proof of Thm is divided into two parts.
Firstly if $X_p(A/L)$ is torsion,
 hence $\fa\times \fb$ is torsion and so
the theorem is a consequence of Proposition
\ref{p:xtor}. Secondly if $X_p(A/L)$ is non-torsion, then the theorem holds trivially, since all items in the equation equal $0$.
\end{proof}

\noindent Remark.
The proof of Theorem \ref{t:xaat} does not use Theorem \ref{t:ct}.

\end{subsection}

%################################################

\subsection{$\fA$ is a $\Gamma$-system}\label{ss:gamma} Now we prove Theorem \ref{t:ct}. Denote
$I_n:=\Ker(\xymatrix{\La\ar@{->>}[r] & \ZZ_p[\Gamma_n]})$.

\begin{lemma}\label{l:longhi} If $\eta\in\La$ is a nonzero element, then
$$\rank_{\ZZ_P} \La/(I_n+(\eta))=\textrm{O}(p^{n(d-1)}).  $$
\end{lemma}
\begin{proof}It follows from \cite[Lemma 1.9.1]{tan10b}, if $\eta$ is irreducible. If $\eta=\eta_1\eta_2$, use the sequence
$$\xymatrix{\La/(I_n+(\eta_1))\ar[r]^{\eta_2} & \La/(I_n+(\eta)) \ar@{->>}[r] & \La/(I_n+(\eta_2))}.$$
\end{proof}

\begin{lemma}\label{l:locfinite} Let $\K$ be a local field of finite residue field. Let $B/\K$ be an abelian variety
and $\K'/\K$ be a finite Galois extension with $G:=\Gal(\K'/\K)$. Then
$\coh^1(G, B(\K'))$ is finite.
\end{lemma}
\begin{proof} We need to show that $B(\K)/\Nm_{G}(B(\K'))$ is finite, because by the local duality of Tate it is
the dual group of $\coh^1(G, B(\K'))$ (see
\cite[Corollary 2.3.3]{tan10a}).

Let $\mathfrak{F}$ denote the formal group law associated to $B/\K$.
Let $\mathfrak{m}$ and $\mathfrak{p}$ denote the maximal ideals of $\K'$ and $\K$.
Let $e$ denotes the ramification index and $\fm^a$ the {\em{different}} of the extension $\K'/\K$.
Suppose $k\geq \frac{2a}{e}+1$, $k\in\ZZ_+$.
Then $\fm^{ek-a}\cdot \fm^{ek-a}$ is contained in the ideal generated by $\fp^{k+1}$.
Also, since $ek-a>0$,
we have $\tr_{G}(\fm^{ek-a})=\fp^k$. Hence by the formal group law,
for each $x\in \fF(\fp^k)$, there exists some
$y_k\in \fF(\fm^{ek-a})$ such that the difference $x_{k+1}:=x-\Nm_{G}(y_k)$ is contained in $\fF(\fp^{k+1})$.
By the same reason, there exists $y_{k+1}$ in $\fF(\fm^{e(k+1)-a})$ such that $x_{k+2}:=x_{k+1}-\Nm_{G}(y_{k+1})$ is contained in $\fF(\fp^{k+2})$. Repeating the argument, we find $y_{k+i}\in \fF(\fm^{e(k+i)-a})$, for each $i$,
so that $x=\Nm_{G}(y)$ for
$y=\sum_{i=0}^\infty y_{k+i}$. This shows $\fF(\fp^k)\subset \Nm_{G}(\fF(\fm^{ek-a}))$.

Let $\bf{B}$ and $\bf{B}'$ be the N$\acute{\text{e}}$ron model of $B$ over $\cO$ and $\cO'$,
the ring of integers of $\K$ and $\K'$ respectively.
Then the identity map on $B$ extends to a unique homomorphism $\bf{B}\longrightarrow \bf{B}'$ that respect the actions of $G$.
The above result implies
$$\fF(\fp^k)\subset \Nm_{G}({\bf{B}}(\cO'))\subset \Nm_{G}({\bf{B}}'(\cO'))=\Nm_{G}(B(\K')).$$
Then the lemma follows, since $B(\K)/\fF(\fp^k)$ is finite.

\end{proof}

For each $n$, denote $\mathcal{H}_n:=\bigoplus_{\text{all}\; w} \coh^1(\Gamma_{n\; w}, A(L_w))$,
the direct sum over all places of $K_n$.

\begin{lemma}\label{l:control} $\mathcal{H}_n$
is cofinitely generated over $\ZZ_p$, and
$$\corank_{\ZZ_p} \mathcal{H}_n= \textrm{O}(p^{n(d-1)}),\;\;\text{as}\;\; n\rightarrow \infty.$$
\end{lemma}
\begin{proof} Write $w\mid S$, if $w$ sits above some place $v\in S$; otherwise, write $w\nmid S$.
Then $\bigoplus_{ w\nmid S} \coh^1(\Gamma_{n\; w}, A(L_w))$ is finite,
by \cite[Lemma 3.2.1]{tan10b}, .

Suppose $w\mid S$. If $A$ has good ordinary reduction at $w$, then $\coh^1(\Gamma_{n\; w}, A(L_w))$ is finite
\cite[Theorem 3]{tan10a}. The same holds, if $A$ becomes to have good ordinary reduction under a finite Galois extension $\K'/K_{n\;w}$,
because then $\coh^1(\K'L_w/\K', A(\K'L_w))$ is finite {\textit{op. cit.}} and the kernel of the composition
$$\xymatrix{\coh^1(\Gamma_{n\; w}, A(L_w))\ar@{^{(}->}[r] & \coh^1(\K'L_w//K_{n\;w},A(\K'L_w))\ar[r] & \coh^1(\K'L_w//\K',A(\K'L_w))}$$
is contained in $\coh^1(\K'/K_{n\;w}, A(\K'))$ that is finite, by Lemma \ref{l:locfinite}.

If $A$ has split multiplicative reduction at some $w$. By \cite[Lemma 3.4.2]{tan10b} as well as its proof, if $g=\dim A$,
we have
$$\corank_{\ZZ_p} \coh^1(\Gamma_{n\; w}, A(L_w))\leq \rank_{\ZZ_p} (K_{n\;w}^\times/\Nm_{L_w/K_{n\;w}}(L_w^\times))^g
=g\cdot \rank_{\ZZ_p}\Gamma_{n\;w} \leq g\cdot d.$$
By applying Lemma \ref{l:locfinite} and the above argument, the same inequality also holds if $A$ has potential multiplicative reduction at $w$.
Then we note that for  $v\in S$, because the decomposition subgroup $\Gamma_v$ is of positive rank, the
number of place of $w$ of $K_n$ sitting over $v$ is of order $O(p^{n(d-1))}$, as $n\rightarrow \infty$.
\end{proof}

%theorem 3 ++++++++++++++++++++++++++++++++++++++++++++++++++++++++++++++++++++
We state our next theorem in the notations of
\ref{sb:cts}.
\begin{mytheorem}\label{t:ct}
Let $K$  be a global field, $L/K$
a $\ZZ_p^d$-extension with a finite ramification locus, and let $A$ be an abelian variety over $K$ which has potentially ordinary reduction over the ramification locus
of $L/K$.
Let $\fA$ be the
Cassels-Tate system of $A$. Then
$\fa$ and $\fb$ are finitely generated torsion
$\La$-modules and
$\fA$  is  a $\Gamma$-system.
\end{mytheorem}
\begin{proof}%[{\bf Proof of Theorem \ref{t:ct}}]
Recall that $Q(\Lambda)$ denotes the fraction field of $\La$. Suppose $\fa$ were non-torsion. Let $r$ and $s$ denote respectively the dimensions over $Q(\La)$ of the vector spaces $Q(\La)\fa$ and $Q(\La)X_p(A/L)$: by \eqref{e:axy}, the former is contained in the latter. Let $e_1,...,e_r,...,e_s\in X_p(A/L)$ form a basis of $Q(\La)X_p(A/L)$ such that $e_1,...,e_r$ are in $\fa$.
Write $\fa'=\Lambda\cdot e_1+\cdots+\Lambda\cdot e_r\subset \fa$ and  $X'=\Lambda\cdot e_1+\cdots+\Lambda\cdot e_s\subset X_p(A/L)$. Then $X_p(A/L)/X'$ is torsion over $\Lambda$, and hence is annihilated by some nonzero $\eta\in \Lambda$.

Let $\omega\in \Gamma^\vee$ be a character not contained in $\Delta_\eta$
(see (\ref{e:Delta})). Extend it
as in \S\ref{su:monsky} to a ring homomorphism $\omega\colon\La \twoheadrightarrow \cO_\omega$ whose kernel we denote by $\ker_\omega$. Then we have the exact sequences
\begin{equation}\label{e:tor1}
\begin{CD} \Tor_{\La} (\La/\ker_\omega, X_p(A/L)/\fa' ) @>>> (\Lambda/\ker_\omega)\otimes_{\La} \fa'@>{id_\omega\otimes i}>> (\La/\ker_\omega)\otimes_{\La} X_p(A/L) , \end{CD}
\end{equation}
where $i:\fa'\longrightarrow X_p(A/L)$ is the inclusion, and
$$\begin{CD} 0 @>>>  X'/\fa'  @>>>  X_p(A/L)/\fa'  @>>> X_p(A/L)/X' @>>> 0 \end{CD}\,.$$
The fact that $X'/\fa'$ is free over $\La$ implies that the natural map
\begin{equation*}\label{e:tor2} \Tor_{\La} (\La/\ker_\omega, X_p(A/L)/\fa' )\, \lr  \Tor_{\La} (\La/\ker_\omega, X_p(A/L)/X' ) \end{equation*}
is an injection. Then the group $\Tor_{\La}(\La/\ker_\omega, X_p(A/L)/\fa' )$ must be finite,
because the quotient $\La/\ker_\omega\simeq \cO_\omega$ is a discrete valuation ring
and $\Tor_{\Lambda} (\La/\ker_\omega, X_p(A/L)/X' )$, being annihilated by
the non-zero residue class of $\eta$ in $\La/\ker_\omega$, is finite. Thus, the homomorphism $id_\omega\otimes i$ in \eqref{e:tor1} must be injective, because $(\Lambda/\ker_\omega)\otimes_{\La} \fa'$ is a free $\cO_\omega$-module.
Hence its image is free of positive rank over $\cO_\omega$.

Now assume that $\omega\in\Gamma_n^\vee\subset \Gamma^\vee$.
Denote $E_n=\QQ_p(\mu_{p^n})$ and $V_n=E_n \otimes_{\ZZ_p}\La/I_n$.
Then $\cO_\omega\subset E_n$ and $\omega$ extends to a ring homomorphism $\omega: V_n\longrightarrow E_n$.
We have
$$V_n=\oplus_{\omega\in\Gamma_n^\vee} V_n^{(\omega)} $$
with $V_n^{(\omega)}=E_n$ and the projection $V_n\longrightarrow V_n^{(\omega)}=E_n$ given by $\omega$.
The above discussion shows if $\omega\not\in \Delta_\eta$, then the image of $id_\omega\otimes i:V_n^{(\omega)}\otimes_{\ZZ_p}\fa'
\longrightarrow V_n^{(\omega)}\otimes_{\ZZ_p} X_p(A/L)$ is a positive dimensional vector space of $E_n$.
Since $\omega\in \Delta_\eta$ if and only if $\omega$ factors through $\La/(I_n+(\eta))$,
in view of Lemma \ref{l:longhi}, we conclude that as $n\rightarrow\infty$, the $\ZZ_p$-rank of the image of
$$\fa' \longrightarrow \La/I_n\otimes_{\La} \fa'\longrightarrow \La/I_n\otimes_{\La} X_p(A/L)$$
is at least of order $p^{dn}+O(p^{n(d-1)})$.
Then the same holds for the image of
$$ \fa\longrightarrow \La/I_n\otimes_{\La} X_p(A/L),$$
since $\fa'\longrightarrow X_p(A/L)$ factors through $\fa'\longrightarrow \fa$. By the duality,
the image of
$$\xymatrix{ \Sel_{p^\infty}(A/L)^{\Gamma_n} \ar[r] & \Sel_{p^\infty}(A/L)/\Sel_{div}(A/L)  }$$
has $\ZZ_p$-corank at least of order $p^{dn}+O(p^{n(d-1)})$. Let $\mathcal{S}(L/K_n)$ denote the preimage
of $\Sel_{p^\infty}(A/L)^{\Gamma_n}$ under the restriction $\coh^1(K_n,A_{p^\infty})\longrightarrow \coh^1(L,A_{p^\infty})$.
Since the composition
$$\xymatrix{\mathcal{S}(L/K_n)\ar[r] & \Sel_{p^\infty}(A/L)^{\Gamma_n} \ar[r] & \Sel_{p^\infty}(A/L)/\Sel_{div}(A/L)  }$$
factors through $\mathcal{S}(L/K_n) \longrightarrow \mathcal{S}(L/K_n)/\Sel_{p^\infty}(A/K_n)_{div}$, while
by Hochschild-Serre spectral sequence and Lemma \ref{l:h1d}, the left morphism has finite cokernel, the
$\ZZ_p$-corank of $\mathcal{S}(L/K_n)/\Sel_{p^\infty}(A/K_n)_{div}$ is at least of order $p^{dn}+O(p^{n(d-1)})$.
Then
the same holds for the
$\ZZ_p$ corank of
$\mathcal{S}(L/K_n)/\Sel_{p^\infty}(A/K_n)$, and hence for that of $\mathcal{H}_n$,
because of the exact sequence
$$0\longrightarrow \Sel_{p^\infty}(A/K_n) \longrightarrow \mathcal{S}(L/K_n) \longrightarrow  \mathcal{H}_n$$
due to the localization map. But Lemma \ref{l:control} says this is absurd.
The proof for $\fb$ just replaces $A$ with $A^t$.
\end{proof}

Theorem \ref{t:ct} actually means  $\fa\times \fb$ is torsion.
In fact it says  $\fa$ is always torsion no matter what happens
to $X_p(A/L)$.
It is worthwhile to mention that Theorem \ref{t:ct} also implies that $\fA$ enjoys property ({\bf T}) (just replace $L/K$ with an arbitrary $\ZZ_p^{d-1}$-subextension $L'/F$).

\subsection{Is $\fA$ pseudo-controlled?} We answer in
Propositions \ref{p:pscontrl} and \ref{p:gdordpscontrl} below.

\begin{lemma} \label{l:zerocoker}
For $n\geq m$ the restriction map
$$\Sel_{p^\infty}(A/K_m)_{div} \longrightarrow (\Sel_{p^\infty}(A/K_n)^{\Gamma^{(m)}})_{div}$$
is surjective.
\end{lemma}

\begin{proof}
The commutative diagram of exact sequences
$$\xymatrix{  \Sel_{p^\infty}(A/K_m)_{div}=\Sel_{p^\infty}(A/K_m)_{div} \ar@{^{(}->}[r]\ar@<-9ex>[d]^{j'} \ar@<9ex>[d]^j & \Sel_{p^\infty}(A/K_m) \ar@{->>}[r]\ar[d]^i & \fb_m \ar[d]^{\fr_m^n}\\
(\Sel_{p^\infty}(A/K_n)^{\Gamma^{(m)}})_{div} \subset (\Sel_{p^\infty}(A/K_n)_{div})^{\Gamma^{(m)}} \ar@{^{(}->}[r] & \Sel_{p^\infty}(A/K_n)^{\Gamma^{(m)}} \ar[r] & \fb_n^{\Gamma^{(m)}} } $$
induces the exact sequence
$$\Ker (\fr_m^n) \longrightarrow \Coker (j)\longrightarrow \Coker (i).$$
Since $\Ker (\fr_m^n)$ is finite while $\Coker (j')$ is $p$-divisible, it is sufficient to show that $\Coker (i)$ is annihilated by some positive integer. Consider the commutative diagram of exact sequences
$$\xymatrix{ \Sel_{p^\infty}(A/K_m) \ar@{^{(}->}[r]\ar[d]^i &  \coh^1_{\mathrm{fl}}(K_m,A_{p^\infty})\ar[r]^-{loc_m} \ar[d]^{res_m^n} & \prod_{\text{all}\;v} \coh^1({K_m}_v, A) \ar[d]^{r_m^n}\\
\Sel_{p^\infty}(A/K_n)^{\Gamma^{(m)}} \ar@{^{(}->}[r] & \coh^1_{\mathrm{fl}}(K_n,A_{p^\infty})^{\Gamma^{(m)}} \ar[r]^-{loc_n} &  \prod_{\text{all}\;w} \coh^1({K_n}_w, A)^{\Gamma_w^{(m)}}}  $$
that induces the exact sequence
$$\Ker \big(\xymatrix{\image (loc_m)\ar[r]^{r_m^n} & \image (loc_n)}\big)\lr \Coker (i)\lr \Coker (res_m^n) \,.$$
By the Hochschild-Serre spectral sequence, the right-hand term $\Coker (res_m^n)$ is a subgroup of $\coh^2(K_n/K_m, A_{p^\infty}(K_n))$, and hence is annihilated by $p^{d(n-m)}=[K_n:K_m]$. Similarly, the left-hand term, being a subgroup of $\prod_v \coh^1({K_n}_v/{K_m}_v, A({K_n}_v))$, is also annihilated by $[K_n:K_m]$.
\end{proof}

\begin{lemma}\label{l:fingen}
If $L/K$ is a $\ZZ_p$-extension and $X_p(A/L)$ is a torsion $\La$-module, then there exists some $N$ such that
$$\Sel_{p^\infty}(A/K_n)_{div}=(\Sel_{p^\infty}(A/K_n)_{div})^{\Gamma^{(m)}}=(\Sel_{p^\infty}(A/K_n)^{\Gamma^{(m)}})_{div}$$
holds for all $n\geq m\geq N$.
\end{lemma}

\begin{proof}
The second equality is an easy consequence of the first one.
The assumption $\Gamma=\gamma^{\ZZ_p}$ implies that if $f\in\La$ is simple then $f$ divides $\gamma^{p^m}-1$ for some $m$. Therefore, by Theorem \ref{t:flat}, there exists an integer $N$ such that $(\gamma^{p^N}-1)\Sel_{div}(A/L)^{\Gamma^{(n)}}=0$ for every $n$. By Lemma \ref{l:h1d} the kernel of the map $\Sel_{p^\infty}(A/K_n)_{div} \rightarrow \Sel_{div}(A/L)^{\Gamma^{(n)}}$ is finite.
This implies that $(\gamma^{p^N}-1)\Sel_{p^\infty}(A/K_n)_{div}$ must be trivial, since it is both finite and $p$-divisible.
\end{proof}

\begin{proposition}\label{p:pscontrl}
If $L/K$ is a $\ZZ_p$-extension and $X_p(A/L)$ is a torsion $\La$-module, then $\fA$ is pseudo-controlled.
\end{proposition}

\begin{proof} We apply Lemmata \ref{l:zerocoker} and \ref{l:fingen}.
If we are given an element $x\in\Sel_{p^\infty}(A^t/K_n)$ with $res_n^l(x)\in\Sel_{p^\infty}(A^t/K_l)_{div}$ for some $l\geq n\geq N$, then we can find $y\in \Sel_{p^\infty}(A^t/K_N)_{div}$ such that $res_N^l(y)=res_n^l(x)$. Then $x-res_N^n(y)\in\Ker(res_n^l)\subset \fs_n^{00}$. This actually shows that $\fa_n^0=\fa_n^{00}$. Then apply Lemma \ref{l:a00}.
\end{proof}

\begin{proposition}\label{p:gdordpscontrl}
If $L/K$ is a $\ZZ_p^d$-extension ramified only at good ordinary places and $X_p(A/L)$ is a torsion $\La$-module, then $\fA$ is pseudo-controlled.
\end{proposition}

\begin{proof} Let $f_1,...,f_m$ be those simple element described in Theorem \ref{t:flat}. and write $g_i:=f_i^{-1} f_1\dots f_m$.
Since $f_i$ divides $\gamma_i^{p^{r_i}}-1$, for some $\gamma_i$ and $r_i$, by Theorem \ref{t:flat} we get
\begin{equation}\label{e:tem01}
g_i\cdot \Sel_{div}(A^t/L)\subset \Sel_{p^\infty}(A^t/L)^{\Psi_i^{(r_i)}},
\end{equation}
where $\Psi_i\subset \Gamma$ is the closed subgroup topologically generated by $\gamma_i$ and we use the notation $H^{(i)}:=H^{p^{r_i}}$ for any subgroup $H<\Gamma$. In view of Proposition \ref{p:pscontrl}, we may assume that $d\geq 2$. Then we can find $\delta_1,...,\delta_d$ as in the proof of Lemma \ref{l:a00} and such that the elements $\delta_jg_i$, $i=1,...,m$, $j=1,...,d$, are coprime. By construction each $\delta_i$ annihilates $\fs_n^{0}$ for all $n$. We are going to show that if $a=(a_n)_n$, $a_n\in \fa_n^0$, is an element in $\fa^0$, then for $n\geq r_i$,
\begin{equation}\label{e:tem02} \delta_jg_i\cdot a_n=0,\;  j=1,...,d. \end{equation}
Then it follows that $\delta_jg_i\cdot \fa^0=0$ for every $i$ and $j$, and hence $\fa^0$ is pseudo-null.

Fix $n\geq r_i$. Choose closed subgroups $\Phi_i\subset\Gamma$, $i=1,...,m$, isomorphic to $\ZZ_p^{d-1}$ such that $\Gamma=\Psi_i\oplus\Phi_i$, and then set $L^{(i)}=L^{\Phi_i^{(n)}}$ and let $L_l^{(i)}$ denote the $l$th layer of the $\ZZ_p$-extension $L^{(i)}/K_n$, so that $\Gal(L^{(i)}/L_l^{(i)})$ is canonically isomorphic to $\Psi_i^{(n)}$.

For each $l\geq n$, let $\xi_l$ be a pre-image of $a_l$ under
$$\xymatrix{\Sel_{p^\infty}(A^t/K_l)\ar[r]^-{\pi_l} & \fa_l}$$
and denote by $\xi'_l$ the image of $\xi_l$ under the corestriction map $\Sel_{p^\infty}(A^t/K_l)\rightarrow \Sel_{p^\infty}(A^t/L^{(i)}_{l-n})$.
Note that $K_l$ is an extension of $L^{(i)}_{l-n}$ because $\Gamma_l\subset \Psi_i^{(l)}\Phi_i^{(n)}=\Gal(L/L^{(i)}_{l-n})$.
Thus, the restriction map sends $\xi'_l$ to an element $\theta_l\in \Sel_{div}(A^t/L)^{\Psi_i^{(l)}\Phi_i^{(n)}}$. Then by \eqref{e:tem01}
\begin{equation}\label{e:tem03} g_i\cdot\theta_l\in \Sel_{div}(A^t/L)^{\Gamma^{(n)}}. \end{equation}
By the control theorem \cite[Theorem 4]{tan10a} the cokernel of the restriction map
$$\xymatrix{\Sel_{p^\infty}(A^t/K_n)_{div}\ar[r]^{res_n} & \Sel_{div}(A^t/L)^{\Gamma^{(n)}}}$$
is finite: we denote by $p^e$ its order. Choose $l\geq n+e$ and choose $\xi_n$ to be the image of $\xi_l$ under the corestriction map $\Sel_{p^\infty}(A^t/K_l)\rightarrow \Sel_{p^\infty}(A^t/K_n)$. Then \eqref{e:tem03} implies that $g_i\cdot \theta_n=p^eg_i\cdot\theta_l$, which shows that $g_i\cdot \theta_n=res_n(\theta_n')$, for some $\theta_n'\in \Sel_{div}(A^t/K_n)$. Then $\pi_n(g_i\cdot \xi_n-\theta'_n)=g_i\cdot a_n$. Since $g_i\cdot \xi_n-\theta'_n\in \fs_n^0$ which is annihilated by $\delta_j$, we have $\delta_j g_i\cdot a_n=0$ as desired.
\end{proof}

\end{section}

%@@@@@@@@@@@@@@@@@@@@@@@@@@@@@@@@@@@@@@@@@@@

\begin{section}{Central Idempotents of the endomorphism ring of $A$} \label{s:al}
Let $\E$ denote the ring of endomorphisms of $A/K$ and write $\ZZ_p\,\E:=\ZZ_p\otimes_{\ZZ}\E$. We assume that there exists a non-trivial idempotent $e_1$ contained in the center of $\ZZ_p\,\E$. Set $e_2:=1-e_1$. Then we have the decomposition:
\begin{equation} \label{e:abstractdec} \ZZ_p\,\E=e_1\ZZ_p\,\E\times e_2\ZZ_p\,\E. \end{equation}

\subsection{The endomorphism rings}\label{ss:idempotents}

Let $\E^t$ denote the endomorphism ring of $A^t/K$. Since the assignment $\psi\mapsto \psi^t$ sending an endomorphism $\psi\in\E$ to its dual endomorphism can be uniquely extended to a $\ZZ_p$-algebra anti-isomorphism $\cdot^t\colon\ZZ_p\,\E\rightarrow \ZZ_p\,\E^t$, we find idempotents $e_1^t,e_2^t$ and the analogue of \eqref{e:abstractdec}.
If $\E$ and $\E^t$ act respectively on $p$-primary abelian groups $M$ and $N$, then these actions can be extended to those of $\ZZ_p\,\E$ and $\ZZ_p\,\E^t$. We have the following $\ZZ_p$-version of Proposition \ref{p:1}.

\begin{lemma}\label{l:psha} For every $a\in\fa_n$, $b\in\fb_n$ and $\psi\in\ZZ_p\,\E$ we have
\begin{equation} \label{e:ctgeneralized} \langle a,\psi_*(b)\rangle_n=\langle \psi_*^t(a), b\rangle_n\,. \end{equation}
\end{lemma}

\begin{proof} First note that any $\psi\in\E$ can be obtained as a sum of two isogenies (e.g., because $k+\psi$ is an isogeny for some $k\in\ZZ$). Thus Proposition \ref{p:1} and linearity of the Cassels-Tate pairing imply that \eqref{e:ctgeneralized} holds for such $\psi$.

In the general case, since $\ZZ_p\,\E$ is the $p$-completion of $\E$, for each positive integer $m$ there exists $\varphi_m\in\E$
such that $\psi-\varphi_m\in p^m\ZZ_p\,\E$. Choose $m$ such that $p^ma=p^mb=0$. Then
$$\langle a,\psi_*(b)\rangle_n=\langle a,{\varphi_m}_*(b)\rangle_n=\langle {\varphi_m^t}_*(a), b\rangle_n=\langle \psi_*^t(a), b\rangle_n.$$
\end{proof}

If $M$ and $N$ are respectively $\ZZ_p\,\E$ and $\ZZ_p\,\E^t$ modules, write $M^{(i)}$ for $e_i\cdot M $ and $N^{(i)}$ for $e_i^t\cdot N$. Then \eqref{e:abstractdec} implies $M=M^{(1)}\oplus M^{(2)}$ and $N=N^{(1)}\oplus N^{(2)}$. In particular,
$$\fa=\fa^{(1)}\oplus \fa^{(2)} \text{ and }\fb=\fb^{(1)}\oplus \fb^{(2)}\,,$$
with $\fa^{(i)}=\varprojlim_{n}\fa^{(i)}_n$ and $\fb^{(i)}=\varprojlim_{n}\fb^{(i)}_n$.

\begin{corollary}\label{c:psha} For every $n$, we have a perfect duality between $\fa_n^{(1)}$, $\fb_n^{(1)}$ and one between $\fa_n^{(2)}$, $\fb_n^{(2)}$.\end{corollary}

\begin{proof}
We just need to check  $\langle \fa_n^{(1)},  \fb_n^{(2)}\rangle_n=\langle \fa_n^{(2)},  \fb_n^{(1)}\rangle_n=0$.
If $a\in \fa_n^{(1)}$ and $b\in \fb_n^{(2)}$, then
$\langle a,  b\rangle_n=\langle e_1^t\cdot a,  e_2\cdot b \rangle_n=\langle  a,  e_1e_2\cdot b\rangle_n=\langle  a, 0\rangle_n=0$.
\end{proof}

Then $\fA^{(i)}:=\{\fa_n^{(i)},\fb_n^{(i)}, \langle\,,\,\rangle,\fr_m^n,\fk_m^n\}$ satisfies conditions ($\Gamma$-1) to ($\Gamma$-4),
and hence by Theorem \ref{t:ct}, are $\Gamma$-systems. They are pseudo-controlled if and only if so is $\fA$. By Theorem \ref{t:al}(3), since $\fA$ enjoys property $(T)$,
and by Proposition \ref{p:gdordpscontrl}, we have the following.

\begin{mytheorem}\label{p:idemsha} If  $\fA$ is pseudo-controlled, then
\begin{equation}\label{e:fapseudo}
[\fa^{(1)}]=[\fb^{(1)}]^{\sharp}\,\;\;\text{and}\;\; \,[\fa^{(2)}]=[\fb^{(2)}]^{\sharp}.
\end{equation}
\end{mytheorem}

\begin{corollary}\label{c:idemsha}
If $L/K$ is a $\ZZ_p^d$-extension ramified only at good ordinary places and $X_p(A/L)$ is a torsion $\La$-module, then
\eqref{e:fapseudo} holds.
\end{corollary}

\subsection{The height pairing} \label{su:d}

Applying $e_i$ to the exact sequence \eqref{e:axy} we get
\begin{equation}\label{e:axy{(1)}} 0\lr \fa^{(i)}\lr X_p(A/L)^{(i)}\lr Y_p(A/L)^{(i)}\lr 0 \,. \end{equation}
Unfortunately in general we are unable to compare either $X_p(A/L)^{(i)}$ with $X_p(A^t/L)^{(i)}$ or $Y_p(A/L)^{(i)}$ with $Y_p(A^t/L)^{(i)}$. However, we can get some partial result as follow.

\subsubsection{The  N\'eron-Tate height pairing}
First we briefly recall the definition of the N\'eron-Tate height pairing
\begin{equation} \label{e:nthp} \tilde h_{A/K}\colon A(K)\times A^t(K)\lr \RR\,. \end{equation}
For details, see \cite[V, \S4]{lan83}. Let
$$P_A\longrightarrow A\times A^t$$
denote the Poincar\'e line bundle: then $\tilde h_{A/K}$ is the canonical height on $A\times A^t$ associated with the divisor class corresponding to $P_A$.

\begin{proposition}\label{p:2}
Let $A$, $B$ be abelian varieties defined over the global field $K$.
Let $\phi\colon A\rightarrow B$ and $\phi^t\colon B^t\rightarrow A^t$ be an isogeny and its dual.
Then the following diagram is commutative:
\[ \begin{CD} \tilde h_{A/K}\colon & A(K) &\,\times & \,A^t(K)  @>>>  \RR\\
 &  @VV{\phi}V @AA{\phi^t}A @| \\
\tilde h_{B/K}\colon & B(K)&\,\times  & \,B^t(K)  @>>> \RR.
\end{CD} \]
\end{proposition}

\begin{proof} By definition of the N\'eron-Tate pairing and functorial properties of the height (\cite[Proposition V.3.3]{lan83}), $\tilde h_{A/K}(\cdot,\phi^t(\cdot))$ and $\tilde h_{B/K}(\phi(\cdot),\cdot)$ are the canonical heights on $A\times B^t$ associated with the divisor classes corresponding respectively to ${(1\times\phi^t)^*(P_A)}$ and $(\phi\times 1)^*(P_B)$. But the theorem in \cite[\S13]{mum74} implies
\begin{equation}\label{e:i}
(1\times \phi^t)^*(P_A)\simeq (\phi\times 1)^*(P_B)
\end{equation}
(see \cite[p. 130]{mum74}).
\end{proof}

\subsubsection{The $p$-adic height pairing} We extend \eqref{e:nthp} to a pairing of $\ZZ_p$-modules.

\begin{lemma}\label{l:pmw}
Let $A$ be an abelian variety defined over the global field $K$.
For every finite extension $F/K$ there exists a $p$-adic height pairing
\begin{equation} \label{e:ntpairing} h_{A/F}\colon \big( \ZZ_p\otimes A(F)\big) \times \big(\ZZ_p \otimes A^t(F)\big) \lr E_F,\end{equation}
where $E_F$ is a finite extension of $\QQ_p$, with the left and right kernels equal to the torsion parts of $\ZZ_p\otimes A(F)$ and $\ZZ_p \otimes A^t(F)$. If $char(K)=p$ one can choose $E_F=\QQ_p$.
\end{lemma}

\begin{proof}
If $char(K)=p$, then after scaling by a factor $\log(p)$, the pairing $\tilde h_{A/F}$ takes values in $\QQ$ (see for example \cite[\S3]{sch82}): in this case we define $h_{A/F}$ by $\tilde h_{A/F}=-\log(p)h_{A/F}$ and extend it to get \eqref{e:ntpairing}. In general, the image of the N\'eron-Tate height $\tilde{h}_{A/F}$ generates a subfield $E_F'\subset \RR$. By the Mordell-Weil Theorem, $E_F'$ is finitely generated over $\QQ$, and hence can be embedded into a finite extension $E_F$ of $\QQ_p$. Then we have the pairing
$$\tilde{h}_{A/F}\colon A(F)\times A^t(F) \lr E_F'\subset E_F\,,$$
which is obviously continuous on the $p$-adic topology, and thus can be extended to a pairing $h_{A/F}$ as required.
Since the left and right kernels of $\tilde{h}_{A/F}$ are the torsion parts of $A(F)$ and $A^t(F)$, if $x_1,...,x_r$ and $y_1,...,y_r$ are respectively $\ZZ$-basis of the free parts of $A(F)$ and $A^t(F)$, then
$${\det}_{i,j}(h_{A/F}(x_i,y_j))={\det}_{i,j}(\tilde{h}_{A/F}(x_i,y_j))\not=0,$$
which actually means that $h_{A/F}$ is non-degenerate on the free part of its domain.
\end{proof}

\subsubsection{}\label{sss:padhpairidem}
For each finite extension $F/K$ write $\M(A/F):=\ZZ_p\otimes A(F)$ and let $h_{A/F}$ be the $p$-adic height pairing established in Lemma \ref{l:pmw}. The action of $\E$ on $A(F)$ extends to that $\ZZ_p\,\E$ on $\M(A/F)$ and the following results are proven by the same reasoning as in the proofs of Lemma \ref{l:psha} and Corollary \ref{c:psha}.

\begin{lemma}\label{l:pmw2}  For every $x\in \M(A/F)$, $y\in \M(A^t/F)$ and $\psi\in\ZZ_p\,\E$ we have
\begin{equation}\label{e:p2mw}
h_{A/F}(\psi(x),y)=h_{A/F}(x,\psi^t(y))\, .
\end{equation}
\end{lemma}

\begin{corollary} \label{c:hdualidemp}
For   $i=1,2$, $\QQ_p\otimes_{\ZZ_p}\M(A/F)^{(i)}$ and $\QQ_p\otimes_{\ZZ_p}\M(A^t/F)^{(i)}$,
 have equal corank over $\ZZ_p$.
\end{corollary}

\subsubsection{} \label{sss:idemZ} Now we assume that all Tate-Shafarevich groups are finite. Hence $\M(A/F)=\Sel_{p^\infty}(A/F)_{div}$
and $\M(A/L):=\varinjlim_{F} \M(A/F)=\Sel_{div}(A/K)$.

The action of $\ZZ_p\,\E$ on $\M(A/L)$ extends to its dual as $(e\cdot\varphi)(x):=\varphi(ex)$.
Write $Y_p(A/L)^{(i)}=(\M(A/L )^{(i)})^\vee$. Then we have
\begin{equation}\label{e:z12} Y_p(A/L)=Y_p(A/L)^{(1)}\oplus Y_p(A/L)^{(2)}. \end{equation}

\begin{mytheorem}\label{p:idemmw} If $X_p(A/L)$ is a torsion $\La$-module, $L/K$ is ramified only at good ordinary places
and $\Sha_{p^\infty}(A/F)$ is finite for every  finite intermediate extension of $L/K$,
then
$$[Y_p(A/L)^{(1)}]=[Y_p(A^t/L)^{(1)}]^\sharp\;\; \text{and}\;\;  [Y_p(A/L)^{(2)}]=[Y_p(A^t/L)^{(2)}]^\sharp.$$
\end{mytheorem}

\begin{proof} Fix $i\in\{1,2\}$. By Theorem \ref{t:flat}, we write
$$[Y_p(A/L)^{(i)}]=\bigoplus_{\nu=1}^m (\La/(f_{\nu}))^{r_{\nu}},\;\;\;[Y_p(A^t/L)^{(i)}]=\bigoplus_{\nu=1}^m (\La/(f_{\nu}))^{s_{\nu}},$$
where $f_1,...,f_m$ are coprime simple elements and $r_{\nu}$, $s_{\nu}$ are nonnegative integers.
We need to show that $r_{\nu}=s_\nu$ for every $\nu$, since $\La/(f_\nu)=(\La/(f_\nu))^\sharp$.

Let $P_1$ denote the quotient $Y_p(A/L)^{(i)}/[Y_p(A/L)^{(i)}]$ and $P_2$ the analogue for $A^t$. Since $P_1$, $P_2$ are pseudo-null $\La$-modules, there are $\eta_1,\eta_2\in\La$ coprime to $f:=f_1 \cdots f_m$ such that $\eta_jP_j=0$. Then $f_\nu$ is coprime to $\eta_1\eta_2ff_\nu^{-1}$
for each $\nu$. We choose $\omega\in \Gamma^\vee$ such that $\omega(f_\nu)=0$ and $\omega(\eta_1\eta_2ff_\nu^{-1}) \not=0$. Let $E$ be a finite extension of $\QQ_p$ containing the values of $\omega$. Write $EM:=E\otimes_{\ZZ_p} M$ for any $M$ over $\ZZ_p$. We see $E$ as a module over $E\La$ via the ring epimorphism $E\La\rightarrow E$ induced by $\omega$. The exact sequence
$$0=\Tor^1_{E\La}(E,EP_1)\lr E\otimes_{E\La} E[Y_p(A/L)^{(i)}] \lr  E\otimes_{E\La} EY_p(A/L)^{(i)} \lr E\otimes_{E\La} EP_1= 0$$
yields
$$r_{\nu}=\dim_{E} (E\otimes_{E\La} EY_p(A/L)^{(i)})\,.$$
Let $\Gamma^\omega\subset\Gamma$ denote the kernel of $\omega$ and write $W_\omega(A)$ for the coinvariants $EY_p(A/L)_{\Gamma^\omega}^{(i)}$. The isomorphisms
$$E\otimes_{E\La} EY_p(A/L)^{(i)}\simeq E\otimes_{E\La} EY_p(A/L)_{\Gamma^\omega}^{(i)}\simeq (EY_p(A/L)_{\Gamma^\omega}^{(i)})^{(\omega)}$$
show that $r_{\nu}=\dim_E W_\omega(A)^{(\omega)}$. A similar argument proves $s_{\nu}=\dim_E W_\omega(A^t)^{(\omega)}$.

Write $K_\omega:=L^{\Gamma^\omega}$ and $\Gamma_{\omega}:=\Gal(K_\omega/K)$. Since
$\Sha_{p^\infty}(A/K_\omega)$ is finite,
the control theorem \cite[Theorem 4]{tan10a} implies that the restriction map $\M(A/K_\omega)\rightarrow \M(A/L)^{\Gamma^\omega}$ has finite kernel and cokernel. Thus we find
$$\rank_{\ZZ_p}\big(\M(A/K_\omega)^{(i)}\big)^\vee=\rank_{\ZZ_p}\big((\M(A/L)^{\Gamma^\omega})^{(i)}\big)^\vee=\rank_{\ZZ_p} Y_p(A/L)_{\Gamma^\omega}^{(i)}.$$
This and Corollary \ref{c:hdualidemp} yield
$$\rank_{\ZZ_p} Y_p(A/L)_{\Gamma^\omega}^{(i)}=\rank_{\ZZ_p} Y_p(A^t/L)_{\Gamma^\omega}^{(i)}.$$
Similarly, for the character $\varpi=\omega^p$, we have
$$\rank_{\ZZ_p} Y_p(A/L)_{\Gamma^{\varpi}}^{(i)}=\rank_{\ZZ_p} Y_p(A^t/L)_{\Gamma^{\varpi}}^{(i)}.$$
As in \S\ref{ss:qorbit}, let $[\omega]$ denote the $\Gal({\bar{\QQ}}_p/\QQ_p)$-orbit of $\omega$. By $\Gamma_\omega^\vee=[\omega]\sqcup\Gamma_{\omega^p}^\vee$ we get the exact sequence of $E[\Gamma_\omega]$-modules:
$$\xymatrix{0 \ar[r] & \prod_{\chi\in [\omega]} (W_\omega(A))^{(\chi)} \ar[r] & W_\omega(A) \ar[r] & W_{\omega^p}(A) \ar[r] & 0}.$$
Since for all $\chi\in [\omega]$ the eigenspaces $(W_{\omega}(A))^{(\chi)}$ have the same dimension over $E$, the two equalities and the exact sequence above imply
$$\dim_E (W_{\omega}(A))^{(\omega)}=\dim_E (W_{\omega}(A^t))^{(\omega)}\,,$$
which completes the proof.\end{proof}

\begin{mytheorem}\label{t:idemx} If $X_p(A/L)$ is a torsion $\La$-module, $L/K$ is ramified only at good ordinary places
and $\Sha_{p^\infty}(A/F)$ is finite for every  finite intermediate extension of $L/K$, then
$$\chi(X_p(A/L))=\chi(X_p(A^t/L)^{(1)})^\sharp\cdot \chi(X_p(A/L)^{(2)})=\chi(X_p(A/L)^{(1)})\cdot \chi(X_p(A^t/L)^{(2)})^\sharp\,.$$
\end{mytheorem}

\begin{proof} Just use the exact sequence \eqref{e:axy{(1)}} (together with its $A^t$-analogue with $\fb^{(i)}$),
 Corollary \ref{c:idemsha} and Theorem \ref{p:idemmw} to get
\begin{equation} \label{e:idemx} \chi(X_p(A^t/L)^{(i)})^\sharp = \chi(X_p(A/L)^{(i)})\,.\end{equation}
\end{proof}

In the next paper \cite{LLTT} of this series we shall apply the theorem to prove
the Iwasawa Main Conjecture for constant ordinary abelian varieties over function fields.

\end{section}


\begin{thebibliography}{D-R}


\bibitem[BL09]{bl09} A.~Bandini and I.~Longhi, {\em{Control theorems for elliptic curves over function fields,}} Int. J. Number Theory {\bf 5} (2009), 229-256.

\bibitem[Bou65]{bou65} N.~Bourbaki, {\em{\'El\'ements de math\'ematique.  Alg\`ebre commutative. Chapitre 7: Diviseurs.}} Hermann, Paris 1965.

\bibitem[Gr03]{gr03} R.~Greenberg, {\em{Galois theory for the Selmer group for an abelian variety,}} Compositio Math. {\bf 136} (2003), 255-297.

\bibitem[Lan83]{lan83} S.~Lang, {\em{Fundamentals of Diophantine geometry,}} Springer-Verlag, Berlin-Heidelberg-New York, 1983.


\bibitem[LLTT]{LLTT}  K.F.~Lai, I.~Longhi, K-S.~Tan, F. Trihan, The Iwasawa Main Conjecture for constant ordinary abelian varieties over function fields, preprint 2013.
\bibitem[MW84]{MW84} B.~Mazur and A.~Wiles, {\em Class fields of abelian extensions of $\mbb{Q}$}, Invent.~Math.  {\bf 76} (1984), no. 2, 179--330.

\bibitem[Mil86]{mil86} J.S.~Milne, {\em{Arithmetic duality theorems,}} Academic Press, New York, 1986.

\bibitem[Mon81]{monsky} P.~Monsky, {\em On $p$-adic power series.}  Math. Ann.  {\bf 255} (1981), no. 2, 217--227.

\bibitem[Mum74]{mum74} D.~Mumford, {\em{Abelian Varieties,}} Oxford Univ. Press, 1974.

    \bibitem[OT09]{ot09} T.~Ochiai and F.~Trihan, {\em{On the Selmer groups of abelian varieties over function fields of characteristic $p>0$,}}  Mathematical Proceedings Cambridge Philosophical Society {\bf 146} (2009) 23-43.

\bibitem[Sch82]{sch82} P. Schneider, {\em Zur Vermutung von Birch und Swinnerton-Dyer \"uber globalen Funktionenk\"orpern}, Math. Ann. {\bf 260} (1982), no. 4, 495--510.

\bibitem[Tan10]{tan10a} K.-S.~Tan, {\em{A generalized Mazur's theorem and its applications,}} Trans. Amer. Math. Soc.  {\bf 362} (2010),  no. 8, 4433--4450.

    \bibitem[Tan12]{tan10b} K.-S.~Tan, {\em Selmer groups over $\ZZ_p^d$-extensions}, to appera Math. Annalen,  arXiv:1205.3907v2.

\end{thebibliography}
\end{document}